\documentclass{scrartcl}
\usepackage{graphicx}
\usepackage[utf8]{inputenc}
\usepackage{amssymb, amsmath, amsthm}
\usepackage[left=3cm, right=3cm]{geometry}
\usepackage{dsfont}
\usepackage{comment}
\usepackage{esint}
\usepackage{tikz}
\usepackage{enumitem}
\usepackage{multicol}
\usepackage{cite}
\usepackage{subcaption}
\usetikzlibrary{decorations.pathreplacing}

\newcommand{\R}{\mathbb{R}}

\newcommand{\1}{\mathbf{1}}
\newcommand{\lm}{\left\|}
\renewcommand{\rm}{\right\|}

\newcommand{\del}{\partial}

\theoremstyle{definition}
\newtheorem{definition}{Definition}[section]
\newtheorem{remark}{Remark}[section]
\theoremstyle{theorem}
\newtheorem{theorem}{Theorem}[section]
\newtheorem{proposition}{Proposition}[section]
\newtheorem{lemma}{Lemma}[section]

\providecommand{\keywords}[1]{\textbf{\textit{Keywords: \,}} #1}
\providecommand{\msc}[1]{\textbf{\textit{2000 Mathematics Subject Classification: \,}} #1}

\begin{document}

\title{Plasticity as the $\mathbf{\Gamma}$-Limit of a Two-Dimensional Dislocation Energy: \\the Critical Regime without the Assumption of Well-Separateness%\thanks{Grants or other notes
}

%\titlerunning{Plasticity as the $\Gamma$-Limit of a Two-Dimensional Dislocation Energy}        % if too long for running head

\author{Janusz Ginster
}


%\institute{J.~Ginster \at
%              Center for Nonlinear Analysis\\
%              Carnegie Mellon University\\
%              5000 Forbes Ave\\
%              Pittsburgh, PA 15213 \\
%              Tel.: +1 412 268 5617\\
%              \email{jginster@andrew.cmu.edu}           %  \\
%%             \emph{Present address:} of F. Author  %  if needed
%           }
%
\date{}

\maketitle

\begin{abstract}
In this paper, a strain-gradient plasticity model is derived from a mesoscopic model for straight parallel edge dislocations in an infinite cylindrical crystal.
The main difference to existing work is that in this work the well-separateness of disloactions is \textit{not} assumed.
In order to prove meaningful lower bounds the ball construction technique, which was developed in the context of Ginzburg-Landau by Jerrard and Sandier, is adapted and modified.
To overcome the difficulty of loss of rigidity on thin annuli during the ball construction a combination of combinatorial arguments and local modifications of the occurring elastic strains is presented.
\end{abstract}
\keywords{$\Gamma$-convergence, plasticity, dislocations, ball construction} \\
\msc{49J45, 58K45, 74C05}

\section{Introduction}

Introduced  in 1907 by Volterra, \cite{Vo07}, dislocations were theoretically found to play a main role to manifest plastic slip in metals by Orowan \cite{Or34}, Polanyi \cite{Po34}, and Taylor \cite{Ta34} in 1934.
%However, findings of dislocations with an electrone microscope were not until 1956, \cite{HiHoWh56}.  \\
Therefore, the derivation of macroscopic plasticity from dislocation models is of tremendeous interest in the mathematical and the mechanical engineering community, \cite{CoGaMu11,GaLePo10,Po07,CoGaOr15,MuScZe14,ScZe12,dLGaPo12,Gr97,Gi17,CeLe05}.
Many of these derivations start from a semi-discrete model where the underlying crystalline lattice is averaged and the dislocations are modelled individually.

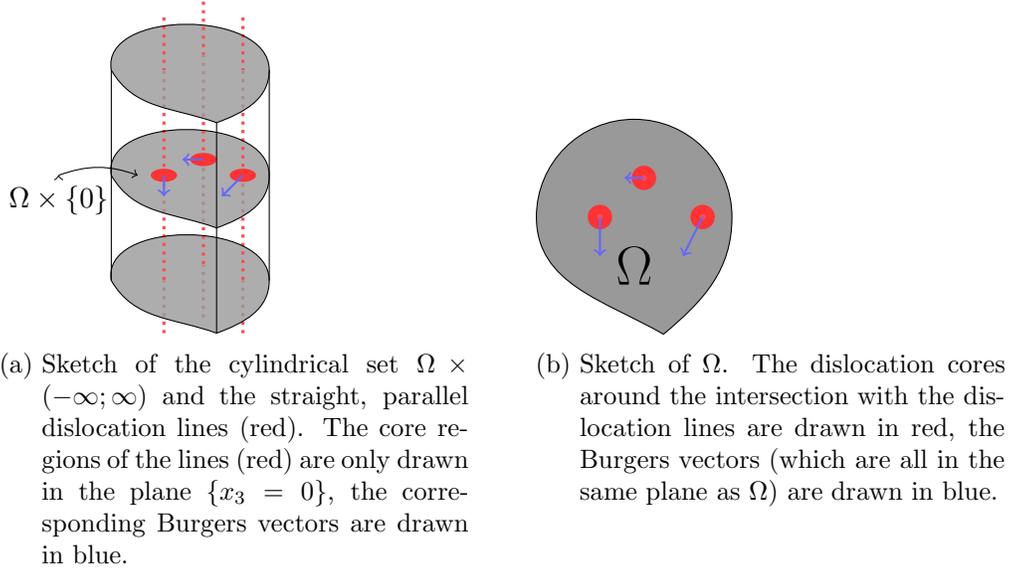
\begin{figure}\centering
\begin{subfigure}[t]{0.4\textwidth}
\begin{tikzpicture}[scale = 0.7]

\draw[dotted, very thick,red!70!white] (1,-3) -- (1,-2);
\draw[dotted, very thick,red!70!white] (2.5,-3) -- (2.5,-2);
\draw[dotted, very thick,red!70!white] (1.75,-2.7) -- (1.75,-1.7);

\fill[gray!80!white, opacity =0.8](0, -2) to [out=100, in=90]   (3,-2) to[out=-90, in =20] (2,-3) to[out=160, in=300]  (0,-2);

\draw[dotted, very thick,red!70!white] (1,0) -- (1,-2);
\draw[dotted, very thick,red!70!white] (2.5,0) -- (2.5,-2);
\draw[dotted, very thick,red!70!white] (1.75,0.3) -- (1.75,-1.7);

\fill[gray!80!white, opacity =0.8](0, 0) to [out=100, in=90]   (3,0) to[out=-90, in =20] (2,-1) to[out=160, in=300]  (0,0);

\draw[dotted, very thick,red!70!white] (1,0) -- (1,2);
\draw[dotted, very thick,red!70!white] (2.5,0) -- (2.5,2);
\draw[dotted, very thick,red!70!white] (1.75,2.3) -- (1.75,0.3);

\fill[gray!80!white, opacity =0.8] (0, 2) to [out=100, in=90]   (3,2) to[out=-90, in =20] (2,1) to[out=160, in=300]  (0,2);

\fill[red!80!white]  plot[smooth,domain=0:2*pi] ({1+cos(\x r)/4},{sin(\x r)/8});
\draw[dotted, very thick,red!70!white] (1,3) -- (1,2);

\fill[red!80!white] plot[smooth,domain=0:2*pi] ({2.5+cos(\x r)/4},{sin(\x r)/8});
\draw[dotted, very thick,red!70!white] (2.5,3) -- (2.5,2);

\fill[red!80!white] plot[smooth,domain=0:2*pi] ({1.75+cos(\x r)/4},{sin(\x r)/8+0.3});
\draw[dotted, very thick,red!70!white] (1.75,3.3) -- (1.75,2.3);

\draw   (0, 0) to [out=100, in=90]   (3,0) to[out=-90, in =20] (2,-1) to[out=160, in=300]  (0,0);
\draw   (0, -2) to [out=100, in=90]   (3,-2) to[out=-90, in =20] (2,-3) to[out=160, in=300]  (0,-2);
\draw   (0, 2) to [out=100, in=90]   (3,2) to[out=-90, in =20] (2,1) to[out=160, in=300]  (0,2);
\draw (0,-2) -- (0,2);
\draw (3,-2) -- (3,2);
\draw (2,-3) -- (2,1);

\draw[->] (-1,0) to[out=20,in=160] (0.5,0);
\draw[->] (-1,0) node[anchor=north] {$\Omega \times \{0\}$};

\draw[thick,blue!60!white,->] (1,0) -- (1,-0.4);
\draw[thick,blue!60!white,->] (1.75,0.3) -- (1.35,0.3);
\draw[thick,blue!60!white,->] (2.5,0) -- (2.1,-0.4);
\end{tikzpicture}
\caption{Sketch of the cylindrical set $\Omega \times (-\infty;\infty)$ and the straight, parallel dislocation lines (red). The core regions  of the lines (red) are only drawn in the plane $\{x_3 = 0\}$, the corresponding Burgers vectors are drawn in blue.}
\end{subfigure}
\quad \quad
\begin{subfigure}[t]{0.4\textwidth}
\begin{tikzpicture}[scale=1.3]
\fill[gray!80!white] (0,0) to[out = 90,in = 180] (1,1) to[out=0, in=90] (2,0) to[out = 270, in=40] (1.3,-1.2) to[out=150,in=270] (0,0);
\draw (0,0) to[out = 90,in = 180] (1,1) to[out=0, in=90] (2,0) to[out = 270, in=40] (1.3,-1.2) to[out=150,in=270] (0,0);
\fill[red!80!white] plot[smooth,domain=0:2*pi] ({0.65+cos(\x r)/8},{sin(\x r)/8});
\fill[red!80!white] plot[smooth,domain=0:2*pi] ({1.1+cos(\x r)/8},{sin(\x r)/8+0.4});
\fill[red!80!white] plot[smooth,domain=0:2*pi] ({1.7+cos(\x r)/8},{sin(\x r)/8});
\fill[red!70!white] plot[smooth,domain=0:2*pi] ({0.65+cos(\x r)/35},{sin(\x r)/35});
\fill[red!70!white] plot[smooth,domain=0:2*pi] ({1.1+cos(\x r)/35},{sin(\x r)/35+0.4});
\fill[red!70!white] plot[smooth,domain=0:2*pi] ({1.7+cos(\x r)/35},{sin(\x r)/35});
\draw (1,-0.5) node[font=\huge] {$\Omega$};
\draw[thick,blue!60!white,->] (0.65,0) -- (0.65,-0.4);
\draw[thick,blue!60!white,->] (1.1,0.4) -- (0.9,0.4);
\draw[thick,blue!60!white,->] (1.7,0) -- (1.5,-0.4);
\end{tikzpicture}
\caption{Sketch of $\Omega$. The dislocation cores around the intersection with the dislocation lines are drawn in red, the Burgers vectors (which are all in the same plane as $\Omega$) are drawn in blue.} \label{fig: 2d}
\end{subfigure}
\caption{Sketch of the geometry in the case of straight, parallel dislocation lines of edge type.} \label{figure: sketchgeometry}
\end{figure}

In this paper, we will restrict ourselves to study an infinite cylindrical crystal and straight, parallel edge dislocations. 
This symmetry allows us to reduce the problem to a plane orthogonal to the dislocations, see Figure \ref{figure: sketchgeometry}.
In fact, the relevant quantities are the in-plane components of the elastic strain and the dislocations are completely characterized by their intersection with this plane.
The presence of dislocations is then modeled by an incompatibility of the planar elastic strain $\beta: \Omega \rightarrow \R^{2\times2}$, \cite{Ny53}, precisely 
\begin{equation}\label{eq: incompatibility}
\operatorname{curl }\beta = \sum_{i} \xi_i \delta_{x_i}
\end{equation}
 for dislocations at intersection points $x_i$ and corresponding Burgers vectors $\xi_i$ which depend on the crystalline lattice and in particular on the interatomic distance $\varepsilon >0$.
 In the case of linearized elasticity the stored elastic energy for a planar elastic strain $\beta$ and a corresponding dislocation density $\mu$ is given by 
 \begin{equation}\label{eq: elasticenergy}
 \int_{\Omega_{\varepsilon}(\mu)} \frac12 \mathcal{C} \beta : \beta \, dx,
 \end{equation}
where $\Omega_{\varepsilon}(\mu) = \Omega \setminus \bigcup_i B_{\varepsilon}(x_i)$ and $\mathcal{C} \in \R^{2\times2\times2\times2}$ is a linear elasticity tensor, i.e., it is positive definite on symmetric matrices.
It is necessary to exclude the discs with radius $\varepsilon$, the so-called cores, around the dislocation in order to regularize the energy.
Indeed, condition \eqref{eq: incompatibility} is incompatible with the space $L^2(\Omega)$.
Precisely, one can show that if $\operatorname{curl }\beta = \xi \delta_0$ then
\begin{equation}\label{eq: lowerboundannulus}
\int_{B_R(0) \setminus B_r(0)} \frac12 \mathcal{C} \beta:\beta \, dx \geq c\left(\frac Rr\right) |\xi|^2 \log\left((\frac Rr \right),
\end{equation}
where $c\left(\frac Rr\right)$ denotes the inverse of Korn's constant for the annulus $B_R(0) \setminus B_r(0)$.\\
In \cite{CeLe05}, the authors derive an asymptotic formula for the energy \eqref{eq: elasticenergy} as they send $\varepsilon \to 0$. 
Garroni, Leoni, and Ponsiglione show in \cite{GaLePo10} that as the number of dislocations $N_{\varepsilon}$ goes to infinity and $\varepsilon \to 0$ the energy \eqref{eq: elasticenergy} features two competing effects, the self-energy of the dislocations and the long-range elastic interaction of the dislocations.
In view of \eqref{eq: lowerboundannulus} the authors compute the limit of the elastic energy rescaled by $N_{\varepsilon} |\log \varepsilon|$ where loosely speaking $N_{\varepsilon}$ is the number
of dislocations in the system.
The limiting behavior of the stored elastic energy depends on the scaling of $N_{\varepsilon}$.
In the subcritical regime, $N_{\varepsilon} \ll |\log \varepsilon|$, the authors derive in the sense of $\Gamma$-convergence a line-tension limit of the form
\begin{equation}\label{eq: limitsubcritical}
\int_{\Omega} \varphi \left( \frac{d\mu}{d|\mu|} \right) \, d|\mu|,
\end{equation}
where $\varphi$ is a subadditive, 1-homogeneous self-energy-density and $\mu$ is the limit of the suitably rescaled dislocation densities. \\
In the supercritical regime, $N_{\varepsilon} \gg |\log \varepsilon|$ the elastic interaction dominates and the $\Gamma$-limit is simply the elastic energy of the limit elastic strain. \\
In the critical regime, $N_{\varepsilon} \approx |\log \varepsilon|$, the $\Gamma$-limit is a strain-gradient model for plasticity (see, for example, \cite{Gu02,FlHu93}) 
\begin{equation}\label{eq: straingradient}
\int_{\Omega} \frac12 \mathcal{C} \beta : \beta \, dx + \int_{\Omega} \varphi\left( \frac{d\mu}{d|\mu|} \right) \, d|\mu|.
\end{equation}
The limit variables $\beta$ and $\mu$ are still coupled through the relation $\operatorname{curl }\beta = \mu$.\\
Note here that for all the results above the authors assume that the dislocations are \textit{well-separated} on an intermediate scale which is much larger than $\varepsilon^{\gamma}$ for any fixed $0 < \gamma < 1$.
This allows to compute the self-energy of all dislocations individually and then relax these in a second step on a larger scale to obtain $\varphi$. \\
In order to obtain compactness for the elastic strains, the authors prove a generalization of Korn's inequality for fields with non-zero curl.\\
Again under the assumption of well-separateness of dislocations, similar results are derived in \cite{MuScZe14,ScZe12,Gi17} for a nonlinear rotationally invariant elastic energy density.
Here, the generalized Korn's inequality is replaced by a generalized version of the Friesecke-James-M\"uller rigidity estimate, \cite{FrJaMu02}.

In \cite{dLGaPo12}, the authors analyze the energy \eqref{eq: elasticenergy} complemented by a core penalization $|\mu|(\Omega)$ without the assumption of well-separateness and derive in the subcritical regime the same line-tension limit as in \eqref{eq: limitsubcritical}. 
Without the well-separateness of dislocations the self-energy cannot be computed for each dislocation individually.
To overcome this problem, the authors adapt the ball-construction technique which was developed in the context of the Ginzburg-Landau functional (see \cite{Je99,Sa98}).
The building block to obtain meaningful lower bounds during the ball-construction technique are energy estimates on annuli.
In view of \eqref{eq: lowerboundannulus}, the main difficulty is that on thin annuli, there is a massive loss of rigidity.
In the subcritical regime, this problem can be overcome by a combinatorial argument which ensures that during a ball construction all occurring annuli have a minimal thickness. \\
Modifications of the ball construction were also used successfully in the context of discrete screw dislocations, \cite{AldLGaPo14}, and to identify a Cosserat-like behavior near grain boundaries, \cite{LaLu16}.

In this paper we study the critical regime \textit{without} the assumption of well-separateness of dislocations.
The same combinatorial argument as in the dilute regime cannot be applied anymore as in the critical regime there are simply too many dislocations in the system to control the thickness of occurring annuli during the ball-construction.
We present a combination of combinatorial arguments and a technique to modify the elastic strain locally to overcome this difficulty (see Proposition \ref{prop: lowerbound1} and Proposition \ref{prop: lowerbound2}).
The main result of this work is the derivation of energy \eqref{eq: lowerboundannulus} from \eqref{eq: elasticenergy} in the sense of $\Gamma$-convergence without well-separateness of dislocations, Theorem \ref{theorem: gammanowell}. 

This article is ordered as follows. 
First, we state the precise mathematical setting of the problem and the main results in Section \ref{sec: setnowell} and Section \ref{sec: mainnowell}, respectively.
In section \ref{section: preliminaries}, we revisit the ball construction technique as it is known, for example, from \cite{Sa98}, and introduce the self-energy.
Next, we prove the key lower bounds for compactness and the $\Gamma$-convergence result in Section \ref{section: keylower}. 
In Section \ref{section: comp}, we prove compactness.
Then, we present the proof of the $\Gamma$-convergence result in the Sections \ref{section: liminf} and \ref{section: limsup}.

\subsection{Setting of the Problem}\label{sec: setnowell}
Throughout this chapter we consider $\Omega \subseteq \R^2$ to be a bounded, simply connected Lipschitz domain which represents the horizontal cross section of a cylindrical crystal.
We denote by $\varepsilon > 0$ the lattice spacing. \\
Moreover, we consider the set of $normalized$ minimal Burgers vectors in the horizontal plane to be $S=\{b_1,b_2\}$ for two linearly independent vectors $b_1,b_2 \in \R^2$. 
The set of (normalized) admissible Burgers vectors is then given by $\mathbb{S} = \operatorname{span}_{\mathbb{Z}} S$.
We consider the following space of admissible dislocation densitites
\begin{equation*}
X(\Omega) = \left\{ \mu \in \mathcal{M}(\Omega;\mathbb{R}^2): \mu = \sum_{i=1}^N \xi_i \delta_{x_i} \text{ for some } N\in \mathbb{N}, 0\neq\xi_i \in \mathbb{S}, \text{ and } x_i \in \Omega \right\}.
\end{equation*}
Note that dealing with a linearized energy density allows us to scale out the dependence of the admissible Burgers vectors from the lattice spacing.
Associated to $\mu \in X(\Omega)$, we consider the strains generating $\mu$.  
As a strain satisfying $\operatorname{curl} \beta = \mu$ for some $\mu \in X(\Omega)$ cannot be in $L^2(\Omega;\mathbb{R}^{2 \times 2})$, we choose a core radius approach, meaning that we consider the reduced domain 
\begin{equation*}
\Omega_{\varepsilon}(\mu) = \Omega \setminus \bigcup_{x \in \operatorname{supp}(\mu)} B_{\varepsilon}(x).
\end{equation*}
In general, we write $\Omega_r(\mu) = \Omega \setminus \bigcup_{x \in \operatorname{supp}(\mu)} B_{r}(x)$ for some $r>0$. \\
The condition \eqref{eq: incompatibility} is then replaced by a circulation condition around the cores.
We define the admissible strains as
\begin{align*}
\mathcal{A}\mathcal{S}_{\varepsilon}(\mu)= \bigg\{ \beta \in L^2(\Omega;\mathbb{R}^{2\times2}): \beta= 0 \text{ in } \Omega \setminus \Omega_{\varepsilon}(\mu), \, \operatorname{curl }\beta = 0 \text{ in } \mathcal{D}'(\Omega_{\varepsilon}(\mu)), \text{ and for every smoothly}
\\ \text{ bounded open set } A \subseteq \Omega \text{ such that } \partial A \subseteq \Omega_{\varepsilon}(\mu) \text{ it holds that } \int_{\partial A} \beta \cdot \tau \,d\mathcal{H}^1 = \mu(A) \bigg\}.
\end{align*}
Here, $\beta \cdot \tau$ has to be understood in the sense of traces, see \cite[Theorem 2]{DaLi88}. 
Note that if the core $B_{\varepsilon}(x_i)$ of a dislocation with Burgers vector $\xi$ does not intersect any other core, the definition of $\mathcal{A}\mathcal{S}_{\varepsilon}$ implies that 
\begin{equation*}
\int_{\partial B_{\varepsilon}(x_i)} \beta \cdot \tau \, d\mathcal{H}^1 = \xi. 
\end{equation*}
Instead of this circulation condition, one could also consider the set $X(\Omega)$ to consist of more regular measures such as 
\begin{equation*}
\frac{\xi }{\pi \varepsilon^2} \, \mathcal{L}^2_{|B_{\varepsilon}(x)}, \, \frac{\xi }{2 \pi \varepsilon} \, \mathcal{H}^1_{|\partial B_{\varepsilon}(x)} \text{ or } \xi \delta_x * \rho_{\varepsilon} \text{ where } \rho_{\varepsilon} \text{ is a standard mollifier}
\end{equation*}
and a strict $\operatorname{curl}$-condition for the admissible strains.
These other possibilities are not equivalent but turn out to produce the same limit energy.
\newline \\
As we focus on the critical regime, we define the rescaled energy $F_{\varepsilon}: \mathcal{M}(\Omega;\mathbb{R}^2) \times L^2(\Omega;\mathbb{R}^{2 \times 2}) \rightarrow [0,\infty]$ as
\begin{equation*}
F_{\varepsilon}(\mu,\beta) = \begin{cases} \frac1{|\log \varepsilon|^2} \left(\int_{\Omega_{\varepsilon}(\mu)} \frac12 \mathcal{C}  \beta: \beta \, dx + |\mu|(\Omega) \right) &\text{if } \mu \in X(\Omega) \text{ and }\beta \in \mathcal{A}\mathcal{S}^{lin}_{\varepsilon}(\mu), \\
+\infty &\text{else},
\end{cases}
\end{equation*}
for an elasticity tensor $\mathcal{C} \in \mathbb{R}^{2\times2\times2\times2}$ which acts only on the symmetric part of a matrix and satisfies 
\begin{equation}\label{eq: elastictensor}
l |F_{sym}|^2 \leq \mathcal{C} F : F \leq L |F_{sym}|^2 \text{ for all } F \in \R^{2\times2}
\end{equation}
for some constants $l,L>0$. \\
Hence, the energy consists of a linearized elastic part and an energy associated to the core of each dislocation.
The core penalization is expected not to contribute in the limit as the dislocation densities are expected to be of order $|\log \varepsilon|$. 
In \cite{Po07} it is shown that in a discrete setting the energy of screw dislocations inside the core is indeed of order $1$.
The same penalization was also used in \cite{dLGaPo12} in the subcritical regime.\\
Finally, we introduce notation for local versions of $X(\Omega), \mathcal{A}\mathcal{S}_{\varepsilon}$, and the energy $F_{\varepsilon}$.
Let $U \subseteq \Omega$ be open.
In the following, we write $X(U)$ for the admissible dislocation densities on $U$ (simply replace $\Omega$ in the definition by $U$).
For $\mu \in X(U)$, we denote by $\mathcal{A}\mathcal{S}_{\varepsilon}(\mu,U)$ the strains generating $\mu$ in $U$ (again replace $\Omega$ by $U$ in the definition of $\mathcal{A}\mathcal{S}_{\varepsilon}$).
Finally, we write $F_{\varepsilon}(\cdot,\cdot,U)$ for the functional defined analogously to $F_{\varepsilon}$ where $\Omega$ is replaced by $U$.

%%%%%%%%%%%%%%%%%%%%%%%%%%%%%%%%%%%%%%%%%%%%%%%%%%%%%%%%

\subsection{Outline}\label{sec: mainnowell}

We define the limit energy $F: \mathcal{M}(\Omega;\mathbb{R}^2) \times L^2(\Omega;\mathbb{R}^{2 \times 2}) \rightarrow [0,\infty]$ as
\begin{equation*}
F(\mu,\beta) = \begin{cases} \int_{\Omega} \frac12 \mathcal{C} \beta: \beta \, dx + \int_{\Omega} \varphi\left(\frac{d\mu}{d|\mu|}\right) \, d |\mu| &\text{if } \mu \in \mathcal{M}(\Omega;\mathbb{R}^2) \cap H^{-1}(\Omega;\mathbb{R}^2), \\
&\beta \in L^2(\Omega;\mathbb{R}^{2 \times 2}), \text{ and } \operatorname{curl }\beta = \mu, \\
+\infty &\text{else }.
\end{cases}
\end{equation*}
Here, $\varphi$ is the relaxed self-energy density which will be defined in \eqref{definition: phi}.
This is the same limit as also obtained in \cite{GaLePo10} for a linearized model \textit{with} the assumption of well-separateness of dislocations.
\\
Before we define the topology which we use to prove the $\Gamma$-convergence of $F_{\varepsilon}$ to $F$, we introduce the flat norm.
Given a measure $\mu \in \mathcal{M}(\Omega;\R^2)$, we define the flat norm by
\begin{equation*}
\lm \mu \rm_{flat} = \sup_{\varphi \in W^{1,\infty}_0(\Omega;\R^2); Lip(\varphi) \leq 1} \int_{\Omega} \varphi \, d\mu.
\end{equation*}
%By the Arzel\`{a}-Ascoli theorem the embedding $W_0^{1,\infty}(\Omega;\R^2) \hookrightarrow C_0^0(\Omega;\R^2)$ is compact. 
%This implies in particular that for a sequence of measures $\mu_k \in \mathcal{M}(\Omega;\R^2)$ converging to $\mu \in \mathcal{M}(\Omega;\R^2)$ in the sense of weak$*$-convergence of measures, it follows the convergence with respect to the flat norm.
Now, we define the convergence of dislocation densities and strains that we use in the $\Gamma$-convergence result.
\begin{definition}\label{def: convlinearized}
Let $\varepsilon_k \to 0$. 
We say that a sequence $(\mu_k,\beta_k)_k \subseteq \mathcal{M}(\Omega;\mathbb{R}^2) \times L^2(\Omega;\mathbb{R}^{2 \times 2})$ converges to $(\mu,\beta) \in \mathcal{M}(\Omega;\mathbb{R}^2) \times L^2(\Omega;\mathbb{R}^{2 \times 2})$ if
\begin{equation*}
\frac{\mu_k}{|\log \varepsilon_k|} \rightarrow \mu \text{ in the flat topology and } \frac{\beta_k}{|\log \varepsilon_k|} \rightharpoonup \beta \text{ in } L^2(\Omega;\mathbb{R}^{2 \times 2}).
\end{equation*}
\end{definition}
\begin{remark}
As the linearized elastic energy only contains the symmetric part of $\beta$, one could expect to define the convergence in this context such that $\frac{(\beta_k)_{sym}}{|\log \varepsilon_k|}$ or in view of Korn's inequality $\frac{\beta_k - W_k}{|\log \varepsilon_k|}$ skew-symmetric matrices $W_k$ converges weakly to $\beta$.
The compactness result will involve a statement of the latter type, see Theorem \ref{thm: compactness}.
However, it is not possible to derive exactly the weak convergence on all of $\Omega$ but only local versions of it.
 As a $\liminf$-inequality is still valid for the convergence of the compactness result, this convergence could be seen as the most natural one for the problem.
Yet, for the sake of a simpler notation, we stick to the convergence defined as above, in particular since the additive appearance of the skew-symmetric matrices leaves no footprint in the limit.
\end{remark}

The main result of this paper is then the following.
\begin{theorem}\label{theorem: gammanowell}
As $\varepsilon \to 0$ the energies $F_{\varepsilon}$ $\Gamma$-converge with respect to the convergence in Definition \ref{def: convlinearized} to $F$.
\end{theorem}

The proof will be given in the Sections \ref{section: liminf} and \ref{section: limsup}. \\
Moreover, we prove the following compactness statement in Section \ref{section: comp}.
\begin{theorem}\label{thm: compactness}
Let $\Omega \subseteq \mathbb{R}^2$ a bounded, connected Lipschitz domain. 
Let $\varepsilon_k \to 0$ and consider a sequence $(\mu_k,\beta_k) \in X(\Omega) \times \mathcal{A}\mathcal{S}_{\varepsilon_k}(\mu_k)$ such that $\sup_k F_{\varepsilon_k}(\mu_k,\beta_k) < \infty$.
Then there exist $\beta \in L^2(\Omega;\mathbb{R}^{2\times2})$, $\mu \in \mathcal{M}(\Omega; \mathbb{R}^{2 }) \cap H^{-1}(\Omega;\mathbb{R}^{2})$, and $W_{k} \in Skew(2)$ such that for a (not relabeled) subsequence it holds
\begin{enumerate}
\item $\frac{\mu_k}{|\log \varepsilon_k|} \rightarrow \mu$ in the flat topology, \label{item: comp1}
\item for all $1 > \gamma > 0$ and $U \subset\subset\Omega$ we have $\frac{\beta_k - W_k}{|\log \varepsilon_k|} \1_{\Omega_{\varepsilon_k^{\gamma}}(\mu_k)} \rightharpoonup \beta$ in $L^2(U;\mathbb{R}^{2 \times 2})$, \label{item: comp2}
\item $\frac{(\beta_k)_{sym}}{|\log \varepsilon_k|} \rightharpoonup \beta_{sym}$ in $L^2(\Omega;\R^{2\times 2})$, \label{item: comp3}
\item $\operatorname{curl } \beta = \mu$. \label{item: comp4}
\end{enumerate}
Finally, it holds
\begin{equation*}
\liminf_{k \to \infty} F_{\varepsilon_k}(\mu_k,\beta_k,\Omega) \geq \int_{\Omega} \mathcal{C} \beta: \beta \,dx + \int_{\Omega} \varphi \left( \frac{d\mu}{d|\mu|} \right) \, d|\mu|.
\end{equation*}
\end{theorem}
\begin{remark}
Notice that we need the localized weak convergence only if we want to control the full strains $\beta_k$.
The symmetric parts $(\beta_k)_{sym}$ clearly converge weakly on the full domain.
\end{remark}

%%%%%%%%%%%%%%%%%%%%%%%%%%%%%%%%%%%%%%%%%%%%%%%%%%%%%%%
\section{Preliminaries}\label{section: preliminaries}
\subsection{The Self-Energy}\label{section: self-energy}

In this section, we define the self-energy density $\varphi$.
For proofs of the statements, we refer to \cite{GaLePo10}. \\
\noindent Let $0 < r_1 < r_2$ and $\xi \in \R^2$. 
We define
\begin{align*}
\mathcal{A}\mathcal{S}_{r_1,r_2}(\xi) = \left\{ \eta \in L^2\left( B_{r_2}(0) \setminus B_{r_1}(0);\R^{2\times 2} \right): \operatorname{curl }\eta = 0  \text{ and } \int_{\del B_{r_1}(0)} \eta \cdot t = \xi \right\}.
\end{align*}
Here, $\tau$ denotes the unit tangent to $\partial B_{r_1}(0)$.
Again, the circulation condition has to be understood in the sense of traces, cf.~\cite[Theorem 2]{DaLi88}. \\
Next, we set
\begin{equation}\label{def:psirR}
\psi_{r_1,r_2}(\xi) = \min \left\{ \frac12 \int_{B_{r_2}(0) \setminus B_{r_1}(0)} \mathcal{C}\eta:\eta \, dx: \, \eta \in \mathcal{A}\mathcal{S}_{r_1,r_2}(\xi) \right\},
\end{equation} 
where $\mathcal{C} = \frac{\del^2 W}{\del^2 F}(Id)$. 
Note that by scaling it holds that $\psi_{r_1,r_2}(\xi) = \psi_{\frac{r_1}{r_2},1}(\xi)$.
The special case $r_2=1$ will be denoted by
\begin{equation}\label{definition: psi}
 \psi(\xi,\delta) = \min \left\{ \frac12 \int_{B_1(0) \setminus B_{\delta}(0)} \mathcal{C}\eta:\eta \, dx: \, \eta \in \mathcal{A}\mathcal{S}_{1,\delta}(\xi) \right\}.
\end{equation}
For fixed $\xi \in \R^2$ the function $\psi(\xi,\delta)$ scales as $|\log \delta|$. 
We state here the following result from \cite[Corollary 6 and Remark 7]{GaLePo10}.

\begin{proposition}\label{prop: convpsi}
 Let $\xi \in \R^2$, $\delta \in (0,1)$ and let $\psi(\xi,\delta)$ be defined as in \eqref{definition: psi}. 
Then for every $\xi \in \R^2$ it holds
 \begin{equation*}
  \lim_{\delta \to 0} \frac{\psi(\xi,\delta)}{|\log \delta|} = \psi(\xi),
 \end{equation*}
where $\psi: \R^2 \rightarrow [0,\infty)$ is defined by
\begin{equation}\label{def: selfenergy}
 \psi(\xi) = \lim_{\delta \to 0} \frac1{|\log \delta|} \frac12 \int_{B_1(0)\setminus B_{\delta}(0)} \mathcal{C} \eta_0: \eta_0 \, dx
\end{equation}
and $\eta_0: \R^2 \rightarrow \R^{2 \times 2}$ is a fixed distributional solution to
\begin{equation*}
 \begin{cases}
  \operatorname{curl }\eta_0 = \xi \delta_0 &\text{ in } \R^2, \\
  \operatorname{div }\eta_0 = 0 &\text{ in } \R^2. 
 \end{cases}
\end{equation*} 
In particular, both limits exist.
Moreover, there exists a constant $K>0$ such that for all $\delta >0$ small enough and $\xi \in \R^2$ it holds
\begin{equation*}
\left| \frac{\psi(\xi,\delta)}{|\log \delta|} - \psi(\xi)\right| \leq K \frac{|\xi|^2}{|\log \delta|}.
\end{equation*}
\end{proposition}
\begin{remark}
Note that the functions $\psi$ and $\psi(\cdot,\delta)$ are $2$-homogeneous and convex.
\end{remark}

The function $\psi$ is the (renormalized) limit self-energy of a single dislocation with Burgers vector $\xi$.
As we will deal with local systems of dislocations, which could lower their self-energy by interaction, we need to relax the self-energy density.

\begin{definition}
We define the function $\varphi: \R^2 \rightarrow [0,\infty)$ by
\begin{equation}\label{definition: phi}
 \varphi(\xi) = \min \left\{ \sum_{k=1}^{M} \lambda_k \psi( \xi_k): \sum_{k=1}^M \lambda_k \xi_k, \, M \in \mathbb{N}, \, \lambda_k \geq 0, \, \xi_k \in \mathbb{S} \right\}.
\end{equation}
\end{definition}
\begin{remark}
Indeed, it can be seen by the $2$-homogeneity of $\psi$ that the $\min$ in the definition of $\varphi$ exists.
\end{remark}
\begin{remark}
 Note that $\varphi$ is convex and $1$-homogenous.
\end{remark}

\subsection{Ball-Construction revisited}\label{sec: ballconstr}

In order to prove compactness or a $\liminf$-inequality, we need to prove bounds for (modified versions of) the dislocation densities $\mu_{\varepsilon}$ in terms of the energy $F_{\varepsilon}$.
The only information we can use is the circulation condition of a corresponding strain $\beta_{\varepsilon} \in \mathcal{A}\mathcal{S}_{\varepsilon}(\mu_{\varepsilon})$.
On a technical level, this circulation condition shares structural properties with the approximation of vortices in the Ginzburg-Landau model. 
A prominent role in proving lower bounds for the Ginzburg-Landau energy play ball constructions, see for example \cite{Je99,Sa98}.
The main ingredient for proving lower bounds, by the use of a ball construction, is a bound of the energy on annuli.
These estimates are based on the fact that a non-zero circulation around an annulus induces a certain minimal amount of energy.
As we deal with linearized elasticity, we control only the symmetric part of the strains.
The use of Korn's inequality allows us to get a lower bound on the energy in terms of the circulation of the strain, see Proposition \ref{lemma: boundcirc}.
As Korn's constant blows up for thin annuli, we need to avoid carefully annuli whose radii go below a certain ratio, see the proofs of Proposition \ref{prop: lowerbound1} and Proposition \ref{prop: lowerbound2}.
In the subcritical regime this can be done by a combinatorial argument, see \cite{dLGaPo12}, which is not valid anymore in the critical regime.
\newline \\
Now, let us shortly review the concept of a \textit{ball construction}.
For a visualization, see Figure \ref{fig: ballconstruction}.
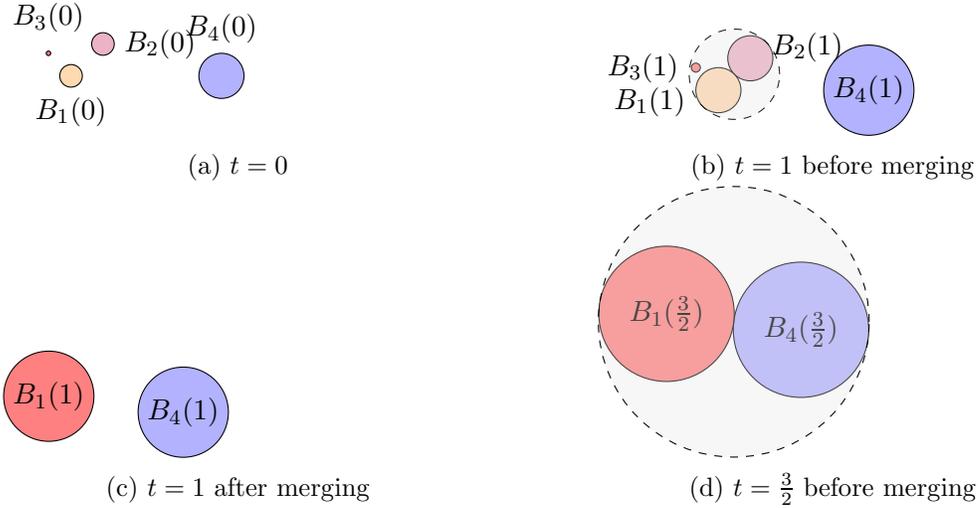
\begin{figure}[h]
\centering
\begin{subfigure}[b]{0.4\textwidth}
\begin{tikzpicture}[scale = 0.3]
\fill[orange!30!white] (0,0) circle (0.5cm);
\fill[blue!30!white] ({sqrt(71/2) + 1/sqrt(2)},0) circle (1cm);
\fill[purple!30!white] ({sqrt(2)},{sqrt(2)}) circle (0.5cm);
\fill[red!50!white] (-1,1) circle (0.1cm);

\draw (0,0) circle (0.5cm);
\draw ({sqrt(2)},{sqrt(2)}) circle (0.5cm);
\draw (-1,1) circle (0.1cm);
\draw ({sqrt(71/2) + 1/sqrt(2)},0) circle (1cm);

\draw ({sqrt(2)+0.5},{sqrt(2)}) node[anchor=west] {$B_2(0)$};
\draw (-1,1.5) node[anchor=south] {$B_3(0)$}; 
\draw (0,-0.5) node[anchor=north] {$B_1(0)$}; 
\draw ({sqrt(71/2) + 1/sqrt(2)},1) node[anchor=south] {$B_4(0)$};
\end{tikzpicture}
\caption{$t=0$}
\end{subfigure}
\qquad  \qquad 
\begin{subfigure}[b]{0.4\textwidth}
\begin{tikzpicture}[scale = 0.3]
\fill[orange!30!white] (0,0) circle (1cm);
\draw (0,0) circle (1cm);

\draw (-1,-0.5) node[anchor=east] {$B_1(1)$};

\fill[purple!30!white] ({sqrt(2)},{sqrt(2)}) circle (1cm);
\draw ({sqrt(2)},{sqrt(2)}) circle (1cm);

\draw ({sqrt(2)+0.5},{sqrt(2)+0.5}) node[anchor=west] {$B_2(1)$};

\fill[red!50!white] (-1,1) circle (0.2cm);
\draw (-1,1) circle (0.2cm);

\draw (-1.3,1) node[anchor=east] {$B_3(1)$};

\fill[blue!30!white] ({sqrt(71/2)+ 1/sqrt(2)},0) circle (2cm);
\draw ({sqrt(71/2)+ 1/sqrt(2)},0) circle (2cm);

\draw ({sqrt(71/2)+ 1/sqrt(2)},0) node {$B_4(1)$};

\draw[dashed] ({1/sqrt(2)},{1/sqrt(2)}) circle (2cm);
\fill[gray!20!white, opacity=0.3] ({1/sqrt(2)},{1/sqrt(2)}) circle (2cm);
\end{tikzpicture}
\caption{$t=1$ before merging}
\end{subfigure}

\begin{subfigure}[b]{0.4\textwidth}
\begin{tikzpicture}[scale = 0.3]
\fill[blue!30!white] ({sqrt(71/2)+ 1/sqrt(2)},0) circle (2cm);
\draw ({sqrt(71/2)+ 1/sqrt(2)},0) circle (2cm);

\draw ({sqrt(71/2)+ 1/sqrt(2)},0) node {$B_4(1)$};
\fill[red!50!white] ({1/sqrt(2)},{1/sqrt(2)}) circle (2cm);
\draw ({1/sqrt(2)},{1/sqrt(2)}) circle (2cm);

\draw ({1/sqrt(2)},{1/sqrt(2)}) node {$B_1(1)$};
\end{tikzpicture}
\caption{$t=1$ after merging}
\end{subfigure}
\qquad  \qquad 
\begin{subfigure}[b]{0.4\textwidth}
\begin{tikzpicture}[scale = 0.3]
\fill[blue!30!white] ({sqrt(71/2)+ 1/sqrt(2)},0) circle (3cm);
\draw ({sqrt(71/2)+ 1/sqrt(2)},0) circle (3cm);
\draw ({sqrt(71/2)+ 1/sqrt(2)},0) node {$B_4(\frac32)$};

\fill[red!50!white] ({1/sqrt(2)},{1/sqrt(2)}) circle (3cm);
\draw ({1/sqrt(2)},{1/sqrt(2)}) circle (3cm);
\draw ({1/sqrt(2)},{1/sqrt(2)}) node {$B_1(\frac32)$};

\fill[gray!20!white, opacity=0.3] ({sqrt(71/8) + 1/sqrt(2)},{1/sqrt(2)/2}) circle (6cm);
\draw[dashed] ({sqrt(71/8) + 1/sqrt(2)},{1/sqrt(2)/2}) circle (6cm);
\end{tikzpicture}
\caption{$t=\frac32$ before merging}
\end{subfigure}

\caption{Sketch of ball construction for four balls with $c=2$.} \label{fig: ballconstruction}
\end{figure}
\\ \underline{\textbf{Ball construction:}}
\newline \\
Fix $c>1$. 
Given a finite family of closed balls $(B_i)_{i \in I}$ with radii $R_i$, we perform the following construction. \newline \\
\underline{\textit{Preparation:}} Find a family of disjoint closed balls $(B_i(0))_{i \in I(0)}$ such that for each $i \in I$ there exists $j \in I(0)$ with the following properties:
$B_i \subseteq B_j(0)$ and $\operatorname{diam}(B_j(0)) \leq \sum_{i: B_i \subseteq B_j(0)} \operatorname{diam}(B_i)$. \\
It is not difficult to see that this is always possible.
\newline \\
\noindent \underline{\textit{Expansion:}} Define for $t>0$ and $i \in I(0)$ the radii $R_i(t)= c^t R_i(0)$ and consider the family of closed balls $(B_i(t))_{i \in I(0)}$ where $B_i(t)$ is the ball with the same center as $B_i(0)$ and radius $R_i(t)$. Moreover, let $I(t) = I(0)$.
We perform this expansion as long as the balls $(B_i(t))_{i \in I(t)}$ are pairwise disjoint. 
For the first $t>0$ such that the family $(B_i(t))_{i \in I(t)}$ is not pairwise disjoint anymore, perform the merging below.
\newline \\
\noindent \underline{\textit{Merging:}} If the family $(B_i(t))_{i \in I(t)}$ is not disjoint, find similarly to the preparation step a disjoint family of balls $(B_j(t))_{j \in J}$ such that for each $i \in I(t)$ there exists an index $j \in J$ which fulfills $B_i(t) \subseteq B_j$ and $\operatorname{diam}(B_j) \leq 2 \sum_{i: B_i(t) \subseteq B_j} R_i(t)$. 
For notational simplicity, let us assume that the index $i \in I(t)$ of a ball $B_i(t)$ that is not affected during the described procedure remains the same i.e., it holds $i \in J$ and $B_i$ is the same ball as $B_i(t)$. \\
Then, replace $I(t)$ by $J$, $(B_i(t))_{i \in I(t)}$ by $(B_j)_{j \in J}$ and the radii $R_i(t)$ by the corresponding $\frac12 \operatorname{diam}(B_j)$. \\
The time $t$ is called a merging time.
After the merging, we continue with the expansion below.
\newline \\
\noindent \underline{\textit{Expansion II:}} Let $\tau > 0$ be a merging time. 
For $t > \tau$, we define the new radii $R_i(t)= c^{t - \tau} R_i(\tau)$ and $I(t) = I(\tau)$.
Moreover, for $i \in I(t)$ set $B_i(t)$ to be the ball with the same center as $B_i(\tau)$ and radius $R_i(t)$. 
Perform this expansion as long as the family $(B_i(t))_{i \in I(t)}$ is disjoint. 
At the first $t > \tau$ such that this is not the case anymore, perform a merging as described above.
\newline \\
We refer to the family $(I(t),(B_i(t))_{i\in I(t)},(R_i(t))_{i\in I(t)})_t$ constructed as above as as the \textit{ball construction} starting with $(B_i)_{i \in I}$ and associated to $c$. \\
By the \textit{discrete ball construction}  starting with $(B_i)_{i \in I}$ and associated to $c$ we mean the discrete subfamily $(I(n),(B_i(n))_{i\in I(n)},(R_i(n))_{i\in I(n)})_{n \in \mathbb{N}}$.
\newline \\
Moreover, we introduce the following notation to link a ball in the construction at time $t$ with its past and future in the construction: For $s_2 > t > s_1 >0$ and $i \in I(t)$ let us define
\begin{align}
P_i^t(s_1) &= \left\{ B_j(s_1): j \in I(s_1) \text{ and } B_j(s_1) \subseteq B_i(t) \right\} \text{ and } \label{eq: defpi} \\
F_i^t(s_2) &= B_j(t_s) \text{ the unique ball at time } s_2 \text{ such that } B_i(t) \subseteq B_j(s_2).
\end{align}
\newline \\
\noindent Next, note the following classical property, c.f.~\cite{Je99}.

\begin{lemma}\label{lemma: ball construction}
Let $(B_i)_{i \in I}$ be a finite family of balls with radii $R_i$ and $c>1$. 
For the corresponding (discrete) ball construction it holds that:
\begin{enumerate}
\item $R_i(t) \leq c^t \sum_{j: B_j \subseteq B_i(t)} R_j$ for all $i \in I(t)$, \label{item: ball construction 1}
\item The construction is monotone in the following sense. Let $t > s \geq 0$. Then for every $i \in I(s)$ there exists $j \in I(t)$ such that $B_i(s) \subseteq B_j(t)$. 
In particular, it holds that $R_i(s) \leq R_j(t)$ and $\bigcup_{i \in I} B_i \subseteq \bigcup_{i \in I(s)} B_i(s) \subseteq  \bigcup_{i \in I(t)} B_i(t)$. \label{item: ball construction 2}
\end{enumerate}
\end{lemma}

\begin{proof}
Property \ref{item: ball construction 1}) is true for $t=0$.
It is easily seen that the the expansion and merging steps preserve this property for growing $t$. 
Property \ref{item: ball construction 2}) is also immediate from the construction.
\end{proof}

%The main idea of the ball construction is to obtain lower bounds during the expansion phase on the growing annuli.
%Hence, estimates on annuli play the role of building blocks for the bounds obtained through the ball construction.
%The main difference to the classical ball construction estimates in Ginzburg-Landau theory is that in the theory of linearized elastic energy there is a significant loss of rigidity on thin annuli which is expressed mathematically by the appearance of Korn\rq{}s constant in the estimate.
The cornerstone of lower bounds via the ball construction are estimates during the merging phase i.e., estimates on annuli.
The following estimate was already proven in \cite{dLGaPo12}. 
We state and prove it here for convenience of the reader.

\begin{lemma} \label{lemma: boundcirc}
Let $R>r>0$ and $\beta \in L^2(B_R(0)\setminus B_r(0);\mathbb{R}^{2\times2})$ such that $\operatorname{curl }\beta = 0$ in $B_R(0) \setminus B_r(0)$.
Then it holds that
\begin{enumerate}
\item $\int_{B_R(0) \setminus B_r(0)} \mathcal{C} \beta: \beta \, dx \geq \frac 1 {K(\frac Rr) \,2 \pi} |\xi|^2 \log \left(\frac R r\right)$, \label{item: boundcirc1}
\item $\int_{B_R(0) \setminus B_r(0)} |\beta|^2 \, dx \geq \frac1{2\pi} |\xi|^2 \log \left(\frac Rr\right)$. \label{item: boundcirc2}
\end{enumerate}
Here, $\xi = \int_{\partial B_r(0)} \beta \cdot \tau \, d\mathcal{H}^1$ where $\tau$ denotes the unit tangent to $\partial B_r(0)$ and $K(\frac R r)$ is Korn's constant for the annulus $B_R(0) \setminus B_r(0)$.
\end{lemma}

\begin{proof}
We may assume that $\beta \in C^0(\overline{B_R(0)\setminus B_r(0)};\R^{2\times2})$.
Korn\rq{}s inequality provides a skew-symmetric matrix $W \in \mathbb{R}^{2 \times 2}$ such that
\begin{equation*}
\int_{B_R(0) \setminus B_r(0)} |\beta - W|^2 \, dx \leq K\left(\frac Rr\right)\int_{B_R(0) \setminus B_r(0)} \mathcal{C} \beta: \beta \, dx.
\end{equation*}
Using a change of variables one can further estimate
\begin{equation*}
\int_{B_R(0) \setminus B_r(0)} |\beta - W|^2 \, dx = \int_r^R \int_{\partial B_t(0)} |\beta - W|^2 \,d\mathcal{H}^1 \, dt \geq \int_r^R \int_{\partial B_t(0)} |(\beta - W)\cdot \tau|^2 \,d\mathcal{H}^1 \, dt.
\end{equation*}
Here, $\tau$ denotes the tangent to the corresponding $\partial B_t(0)$. \\
Jensen\rq{}s inequality yields
\begin{equation*}
\int_r^R \int_{\partial B_t(0)} |(\beta - W)\cdot \tau|^2 \,d\mathcal{H}^1 \, dt \geq  \int_r^R \frac1{2\pi t} \left| \int_{\partial B_t(0)} (\beta - W)\cdot \tau \,d\mathcal{H}^1\right|^2 \, dt = \frac1{2\pi} \log\left(\frac Rr\right) |\xi|^2.
\end{equation*}
Combining the estimates, we find \eqref{item: boundcirc1}.
The last two estimates for $W = 0$ show \eqref{item: boundcirc2}.
\end{proof}
\begin{remark}
It can be seen that $K(\frac{R}{r}) \rightarrow \infty$ as $\frac Rr \to 0$, see \cite{PaWe61}.
\end{remark}
%%%%%%%%%%%%%%%%%%%%%%%%%%%%%%%%%%%

\section{Proofs of the Main Results}

\subsection{The Main Ingredients for Lower Bounds}\label{section: keylower}

The main difficulty in a regime with more than $|\log \varepsilon|$ dislocations is that in a ball construction argument one cannot avoid the combinatorics of distinguishing balls which expand for a certain minimal time and therefore induce a relevant energy to the system from those that merge so frequently that they do not allow to estimate their corresponding energy uniformly due to the blow-up of Korn's constant on thin annuli. \\
In the following proposition, we show how to reduce the general situation in the critical regime to a situation that is easier to analyze. 
Essentially, we prove that in a neighborhood of the dislocations of order $\varepsilon^{\gamma}$ we can change the strain $\beta_{\varepsilon}$ slightly. 
The total variation of the $\operatorname{curl}$ of the new strain $\bar{\beta}_{\varepsilon}$ is controlled in terms of $|\log \varepsilon|$ and the $\operatorname{curl}$ is concentrated in at most $C |\log \varepsilon|$ balls with a radius that is much smaller than $\varepsilon^{\gamma}$ for some fixed $0 < \gamma < 1$.

\begin{proposition}\label{prop: lowerbound1}
Let $1> \alpha > \gamma > 0$, $\delta \geq 0$, and $K>0$. 
There exists $\varepsilon_0 = \varepsilon_0(\alpha,\gamma,K)>0$ such that for all $0 < \varepsilon < \varepsilon_0$ it holds the following: \\
Let $A_{\varepsilon} \subseteq \mathbb{R}^2$ open.
Let $\mu_{\varepsilon} \in X(A_{\varepsilon})$ and $\beta_{\varepsilon} \in \mathcal{A}\mathcal{S}_{\varepsilon}(\mu_{\varepsilon}, A_{\varepsilon})$ such that $\operatorname{dist}(\operatorname{supp}(\mu_{\varepsilon}), \partial A_{\varepsilon}) \geq \varepsilon^{\gamma}$ and $F_{\varepsilon}(\mu_{\varepsilon},\beta_{\varepsilon},A_{\varepsilon}) \leq K |\log \varepsilon|^{-\delta}$.
Then there exist a family of disjoint closed balls $(D_i^{\varepsilon})_{i \in I_{\varepsilon}}$ and a function $\bar{\beta}_{\varepsilon}: A_{\varepsilon} \rightarrow \mathbb{R}^{2 \times 2}$ such that 
\begin{enumerate}[label=(\roman*)]
\item $\operatorname{diam} D_i^{\varepsilon} \leq  \varepsilon^{\alpha}$ and $D_i^{\varepsilon} \cap \operatorname{supp }(\mu_{\varepsilon}) \neq \emptyset$ for all $i \in I_{\varepsilon}$, \label{item: lowerbound11}
\item $|I_{\varepsilon}| \leq C(\alpha,K) |\log \varepsilon|^{1-\delta}$, \label{item: lowerbound12}
\item $\bar{\beta}_{\varepsilon} = \beta_{\varepsilon}$ on $A_{\varepsilon} \setminus \bigcup_{ x \in \operatorname{supp }(\mu_{\varepsilon})} B_{\varepsilon^{\alpha}}(x)$, \label{lowerbound: lowerbound12ab}
\item $\operatorname{curl } \bar{\beta}_{\varepsilon} \in \mathcal{M}(A_{\varepsilon};\R^2)$, \label{item: lowerbound 13}
\item $\operatorname{supp }(\operatorname{curl }\bar{\beta}_{\varepsilon}) \subseteq \bigcup_{i \in I_{\varepsilon}} D_i^{\varepsilon}$  and $(\operatorname{curl } \bar{\beta}_{\varepsilon}) (U) = \mu_{\varepsilon}(U)$ for each connected component $U$ of $A_{\varepsilon}$, \label{item: lowerbound14}
\item $|\operatorname{curl} \bar{\beta}_{\varepsilon}|(A_{\varepsilon})| \leq C(\alpha,K) |\log \varepsilon|^{1 - \delta}$, \label{item: lowerbound15}
\item $ \frac1{|\log \varepsilon|^2} \int_{A_{\varepsilon}} \frac12 \mathcal{C} \bar{\beta}_{\varepsilon} : \bar{\beta}_{\varepsilon} \, dx \leq F_{\varepsilon}(\mu_{\varepsilon},\beta_{\varepsilon},A_{\varepsilon}) + C(\alpha,K) \frac{ F_{\varepsilon}(\mu_{\varepsilon},\beta_{\varepsilon},A_{\varepsilon})}{|\log \varepsilon|}$. \label{item: lowerbound1last}
\end{enumerate}
\end{proposition}

\begin{proof}
Let $\sigma=\frac{1 - \alpha}3$ and fix $c>1$. Let $\varepsilon >0$.\\ 
The prove is subdivided in three steps. 
It is based on a ball construction starting with the balls of radius $\varepsilon$ around the dislocation points.
First, we estimate the number of balls, whose $\mu_{\varepsilon}$-measure is non-zero, at some time in the ball construction that corresponds to balls with an intermediate radius $\varepsilon^{\alpha + 2 \sigma}$.
Secondly, at a later point in the construction we bound the number of balls whose accumulated Burgers vector is zero by deleting dipoles without creating too much energy nor changing the strains on a large set, see \ref{item: lowerbound1last} and \ref{lowerbound: lowerbound12ab}.
Combining the estimates leads to \ref{item: lowerbound12}.
In a third step, we modify the strains slightly in order to obtain a strain with a $\operatorname{curl}$ that is still related to $\mu_{\varepsilon}$ but whose total variation is bounded in terms of $|\log \varepsilon|$, see \ref{item: lowerbound14} and \ref{item: lowerbound15}.
\newline \\  
\textbf{Step 1.} \emph{Estimation of number of balls such that $\mu_{\varepsilon}(B) \neq 0$.} \\
Let $B_i^{\varepsilon} = B_{\varepsilon}(x_i)$ where $\operatorname{supp }\mu_{\varepsilon} = \{x_1, \dots, x_{N_{\varepsilon}}\}$.
As the elements in $\mathbb{S}$ are bounded away from zero, we may deduce from the assumed energy bound that $N_{\varepsilon} \leq k |\mu_{\varepsilon}| \leq k K |\log \varepsilon|^{2-\delta}$. \\
Now, perform a continuous ball construction starting with the balls $(B_i^{\varepsilon})_{i=1,\dots,N_{\varepsilon}}$ and denote its output by $(I_{\varepsilon}(t), (B_i^{\varepsilon}(t))_t, (R_i^{\varepsilon}(t))_t)_t$.
In this first step, we consider only times $t>0$ such that $\sum_{i \in I_{\varepsilon}(t)} R^{\varepsilon}_i(t) \leq \varepsilon^{\alpha + 2 \sigma}$.
Using  Lemma \ref{lemma: ball construction}, we can compute a lower bound on $t^{\varepsilon}_1 > 0$ which we define to be the first time $t$ such that $\sum_{i \in I_{\varepsilon}(t)} R^{\varepsilon}_i(t) = \varepsilon^{\alpha + 2 \sigma}$:
\begin{equation*}
\varepsilon^{\alpha + 2 \sigma} = \sum_{i \in I_{\varepsilon}(t_1^{\varepsilon})} R^{\varepsilon}_i(t) \leq c^{t_1^{\varepsilon}} \sum_{i \in I_{\varepsilon}(t_1^{\varepsilon})} \varepsilon \, \#\{ B_j^{\varepsilon}: B_j^{\varepsilon} \subseteq B_i(t_1^{\varepsilon}) \} \leq kK \, c^{t^{\varepsilon}_1} \varepsilon |\log \varepsilon|^2.
\end{equation*}
From this estimate one derives directly 
\begin{equation*}
t^{\varepsilon}_1 \geq \sigma \frac{|\log \varepsilon|}{\log c} - \frac{\log (k K) -2 \log |\log \varepsilon|}{\log c}.
\end{equation*}
In particular, for $\varepsilon > 0$ small enough (depending on $K$ and $\sigma$) we obtain that $t^{\varepsilon}_1 \geq \frac{\sigma}2 \frac{|\log \varepsilon|}{\log c} + 1$. 
Let us consider the balls $(B^{\varepsilon}_i(s^{\varepsilon}_1))_{i \in I(s^{\varepsilon}_1)}$ of the the ball construction at time $s^{\varepsilon}_1 = \lceil \frac{\sigma}2 \frac{|\log \varepsilon|}{\log c} \rceil$. 
Note that all balls in $(B^{\varepsilon}_i(s^{\varepsilon}_1))_{i \in I(s^{\varepsilon}_1)}$ have a radius which is smaller than $\varepsilon^{\alpha + 2\sigma}$. \\
We subdivide the family of balls  $(B^{\varepsilon}_i(s^{\varepsilon}_1))_{i \in I_{\varepsilon}(s^{\varepsilon}_1)}$ into the subset of balls that evolve from few mergings and those that originate from many mergings:
\begin{align*}
\mathcal{F}_{\varepsilon}(s^{\varepsilon}_1) = \bigg\{ B^{\varepsilon}_i(s^{\varepsilon}_1): \# P_i^{s_1^{\varepsilon}}(0) \leq \frac{s_1^{\varepsilon}}{2} \bigg\} \text{ and }
\mathcal{M}_{\varepsilon}(s^{\varepsilon}_1)  = \bigg\{ B^{\varepsilon}_i(s^{\varepsilon}_1): \# P_i^{s_1^{\varepsilon}}(0) > \frac{s_1^{\varepsilon}}{2} \bigg\}.
\end{align*}
Recall that by definition the set $P_i^{s_1^{\varepsilon}}(0)$ contains the balls at time zero which are included in the ball $B_i^{\varepsilon}(s_1^{\varepsilon})$, see \eqref{eq: defpi}. \\
Let us first estimate the number of balls in $\mathcal{M}_{\varepsilon}(s^{\varepsilon}_1)$. 
By definition, every ball $B^{\varepsilon}_i(s^{\varepsilon}_1) \in \mathcal{M}_{\varepsilon}(s^{\varepsilon}_1)$ originates from at least $\frac{s_1^{\varepsilon}}2$ starting balls. 
Consequently, 
\begin{equation}
\# \mathcal{M}_{\varepsilon}(s^{\varepsilon}_1) \leq 2 \frac{N_{\varepsilon}}{s^{\varepsilon}_1} \leq 2 k K \frac{|\log \varepsilon|^{2-\delta}}{\frac{\sigma}2 \frac{|\log \varepsilon|}{\log c}} = \frac{4 k K}{\sigma} \log(c) \, |\log \varepsilon|^{1-\delta}. \label{eq: estM}
\end{equation}
The next objective is to estimate the number of balls in $\mathcal{F}_{\varepsilon}(s^{\varepsilon}_1)$ which have an accumulated Burgers vector which is non-zero.\\
Fix $B^{\varepsilon}_i(s^{\varepsilon}_1) \in \mathcal{F}_{\varepsilon}(s^{\varepsilon}_1)$.
By definition of $\mathcal{F}_{\varepsilon}(s_1^{\varepsilon})$, the ball $B_i^{\varepsilon}(s_1^{\varepsilon})$ includes at most $\frac{s_1^{\varepsilon}}2$ starting balls.
Hence, it evolves from at most $\frac{s_1^{\varepsilon}}2$ mergings.
By the pigeonhole principle, there exist natural numbers $0\leq n_1 < \dots < n_{L_i} \leq s_1^{\varepsilon}-1$ such that for all $k= 1 ,\dots, L_i \geq \lfloor \frac{s_1^{\varepsilon}}2 \rfloor$ every ball $B^{\varepsilon}_j(n_k) \in P_i^{s^{\varepsilon}_1}(n_k)$ purely expands in the time interval $(n_k,n_k +1]$.
Recall that by the definition of the ball construction this means that $B^{\varepsilon}_j(n_k +1)$ has the same center as $B^{\varepsilon}_j(n_k)$ and the radius $R^{\varepsilon}_j(n_k +1) = c R^{\varepsilon}_j(n_k)$.
Moreover, we know that $\operatorname{curl } \beta_{\varepsilon} = 0$ in $B_j(n_k + 1) \setminus B_j(n_k)$ (remember that $\operatorname{supp }\mu_{\varepsilon} \subseteq \bigcup_{i \in I_{\varepsilon}} B^{\varepsilon}_i \subseteq \bigcup_{i \in I_{\varepsilon}(t)} B^{\varepsilon}_i(t)$ for all $t>0$).
Hence, we can apply Lemma \ref{lemma: boundcirc} to $B^{\varepsilon}_j(n_k +1) \setminus B^{\varepsilon}_j(n_k)$ to obtain, by summing over all these disjoint annuli,
\begin{align}
\int_{B^{\varepsilon}_i(s^{\varepsilon}_1)} \mathcal{C} \beta_{\varepsilon} : \beta_{\varepsilon} \, dx  \geq &\sum_{k=1}^{L_i} \, \sum_{B^{\varepsilon}_j(n_k) \in P_i^{s^{\varepsilon}_1}(n_k)} \int_{B^{\varepsilon}_j(n_k + 1) \setminus B^{\varepsilon}_j(n_k)} \mathcal{C} \beta_{\varepsilon} : \beta_{\varepsilon} \,dx \label{eq: lowerboundstep11} \\
\geq &\frac{1}{2 \pi K(c)}\log(c)\sum_{k=1}^{L_i} \sum_{B^{\varepsilon}_j(n_k) \in P_i^{s^{\varepsilon}_1}(n_k)} |\mu_{\varepsilon}(B^{\varepsilon}_j(n_k))|^2. \label{eq: adas}\\
\intertext{As $\mu_{\varepsilon}(B^{\varepsilon}_j(n_k)) \in \mathbb{S}$ and the non-zero elements in $\mathbb{S}$ are bounded away from zero, we can further estimate}
\eqref{eq: adas} \geq &\frac{k}{2 \pi K(c)} \log(c)\sum_{k=1}^{L_i} \sum_{B^{\varepsilon}_j(n_k) \in P_i^{s^{\varepsilon}_1}(n_k)} |\mu_{\varepsilon}(B^{\varepsilon}_j(n_k))| \nonumber  \\
 \geq &\frac{k}{2 \pi K(c)}\log(c) \sum_{k=1}^{L_i} |\mu_{\varepsilon}(B^{\varepsilon}_i(s^{\varepsilon}_1))| \nonumber \\
\geq &\frac{k }{2 \pi K(c)}\log(c) \,  L_i \, |\mu_{\varepsilon}(B^{\varepsilon}_i(s^{\varepsilon}_1))| \nonumber  \\ 
\geq &\frac{k}{2 \pi K(c)} \log(c) \, \left\lfloor \frac{s_1^{\varepsilon}}2 \right\rfloor  |\mu_{\varepsilon}(B^{\varepsilon}_i(s^{\varepsilon}_1))| \nonumber \\
\geq &\frac{\sigma \, k}{16 \pi K(c)} |\log \varepsilon| \,|\mu_{\varepsilon}(B^{\varepsilon}_i(s^{\varepsilon}_1))|. \label{eq: lowerboundstep1last}
\end{align}
For the second inequality, we used that $\mu_{\varepsilon}(B_i^{\varepsilon}(s_1^{\varepsilon})) = \sum_{B^{\varepsilon}_j(n_k) \in P_i^{s^{\varepsilon}_1}(n_k)} \mu_{\varepsilon}(B^{\varepsilon}_j(n_k))$; 
for the last inequality we used that for $\varepsilon$ small enough it holds that $\left\lfloor \frac{s_1^{\varepsilon}}2 \right\rfloor \geq  \frac{\sigma |\log \varepsilon|}{8 \log c}$. \\
By summing over all $B^{\varepsilon}_i(s^{\varepsilon}_1) \in \mathcal{F}_{\varepsilon}(s^{\varepsilon}_1)$, we deduce from the energy bound on $F_{\varepsilon}(\mu_{\varepsilon},\beta_{\varepsilon},A_{\varepsilon})$ that
\begin{equation*}
\sum_{B^{\varepsilon}_i(s^{\varepsilon}_1) \in \mathcal{F}_{\varepsilon}(s^{\varepsilon}_1)} |\mu_{\varepsilon}(B^{\varepsilon}_i(s^{\varepsilon}_1))| \leq \frac{\sigma \, k}{16 \pi K(c)} K |\log \varepsilon|^{1 - \delta}.
\end{equation*}
In particular, we obtain the bound (recall that the non-zero elements of $\mathbb{S}$ are bounded away from zero) 
\begin{equation}
\#\{B^{\varepsilon}_i(s^{\varepsilon}_1) \in \mathcal{F}_{\varepsilon}(s^{\varepsilon}_1): \mu_{\varepsilon}(B^{\varepsilon}_i(s^{\varepsilon}_1)) \neq 0 \} \leq C(\alpha,K,c) |\log \varepsilon|^{1 - \delta}. \label{eq: estF}
\end{equation}
Combining the bounds \eqref{eq: estM} and \eqref{eq: estF} provides the estimate
\begin{equation*}
\# \{ B^{\varepsilon}_i(s^{\varepsilon}_1): i \in I_{\varepsilon}(s^{\varepsilon}_1) \text{ and } \mu_{\varepsilon}(B^{\varepsilon}_i(s^{\varepsilon}_1)) \neq 0 \} \leq \tilde{C}(\alpha,K,c) |\log \varepsilon|^{1-\delta}.
\end{equation*}
\newline \\
\textbf{Step 2. } \emph{Reduction of number of balls such that $\mu_{\varepsilon}(B) = 0$.} \\
In this step, we reduce the number of balls such that $\mu_{\varepsilon}(B)=0$ by further growing the balls from step 1 and replacing $\beta_{\varepsilon}$ by local gradients on balls with $\mu_{\varepsilon}$-mass $0$. \\
Let $\tilde{I}_{\varepsilon} = I_{\varepsilon}(s^{\varepsilon}_1)$ and $\tilde{B}^{\varepsilon}_i = B^{\varepsilon}_i(s^{\varepsilon}_1)$ for all $i \in \tilde{I}_{\varepsilon}$ where $I_{\varepsilon}(s^{\varepsilon}_1)$ and $B^{\varepsilon}_i(s^{\varepsilon}_1)$ are from step 1.
Consider a new ball construction associated to $c$ starting with the balls $(\tilde{B}^{\varepsilon}_i)_{i \in \tilde{I}_{\varepsilon}}$. 
With a little abuse of notation we call the output of this ball construction again $(I_{\varepsilon}(t), (B_i^{\varepsilon}(t))_{i \in I_{\varepsilon}(t)}, (R_i^{\varepsilon})_{i \in I_{\varepsilon}(t)})_t$. \\
As the starting balls have --- by construction in step 1 --- a radius less than $\varepsilon^{\alpha + 2 \sigma}$, we can argue as in step 1 to obtain that for $\varepsilon >0$ small enough the inequality $\sum_{i \in I_{\varepsilon}(t)} R^{\varepsilon}_i(t) \leq \varepsilon^{\alpha + \sigma}$ holds true for all $t \leq \lceil \frac{\sigma}2 \frac{|\log \varepsilon|}{\log c} \rceil =: s^{\varepsilon}_2$. \\
We define the following partition of the set $\{B^{\varepsilon}_i(s^{\varepsilon}_2): i \in I_{\varepsilon}(s^{\varepsilon}_2)\}$ (see Figure \ref{fig: Aeps}):
\begin{align*}
A^{\varepsilon}_1(s^{\varepsilon}_2) = \bigg\{&B^{\varepsilon}_i(s^{\varepsilon}_2): \text{ there exists a ball } B^{\varepsilon}_j(0) \in P_i^{s_2^{\varepsilon}}(0) \text{ such that } \mu_{\varepsilon}(B^{\varepsilon}_j(0)) \neq 0  \bigg\},  \\
A^{\varepsilon}_2(s^{\varepsilon}_2) = \bigg\{&B^{\varepsilon}_i(s^{\varepsilon}_2): \text{for all } B^{\varepsilon}_j(0) \in P_i^{s_2^{\varepsilon}}(0) \text{ it holds } \mu_{\varepsilon}(B^{\varepsilon}_j(0)) = 0 \text{ and } \# P_i^{s_2^{\varepsilon}}(0) > \frac{s_2^{\varepsilon}}2 \bigg\},  \\
A^{\varepsilon}_3(s^{\varepsilon}_2) = \bigg\{&B^{\varepsilon}_i(s^{\varepsilon}_2): \text{for all } B^{\varepsilon}_j(0) \in P_i^{s_2^{\varepsilon}}(0) \text{ it holds } \mu_{\varepsilon}(B^{\varepsilon}_j(0)) = 0 \text{ and } \# P_i^{s_2^{\varepsilon}}(0) \leq \frac{s_2^{\varepsilon}}2 \bigg\}.
\end{align*}

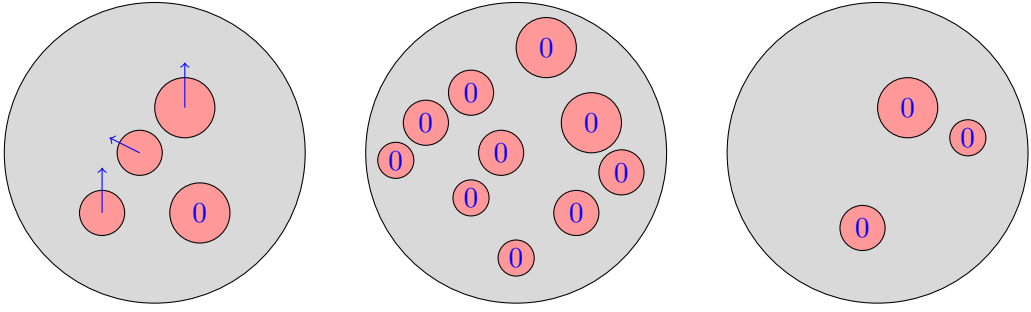
\begin{figure}[t]
\centering
\begin{subfigure}{0.3\textwidth}
\begin{tikzpicture}[scale=2]
\fill[gray!30!white] (0,0) circle (1cm);
\draw (0,0) circle (1cm);
\fill[red!40!white] (0.2,0.3) circle (0.2cm);
\draw (0.2,0.3) circle (0.2cm);
\draw[blue,->] (0.2,0.3) -- (0.2,0.6);

\fill[red!40!white] (-0.1,0) circle (0.15cm);
\draw (-0.1,0) circle (0.15cm);
\draw[blue,->] (-0.1,0) -- (-0.3,0.1);

\fill[red!40!white] (0.3,-0.4) circle (0.2cm);
\draw (0.3,-0.4) circle (0.2cm);
\draw[blue] (0.3,-0.4) node {$0$};

\fill[red!40!white] (-0.35,-0.4) circle (0.15cm);
\draw (-0.35,-0.4) circle (0.15cm);
\draw[blue,->] (-0.35,-0.4) -- (-0.35,-0.1);
\end{tikzpicture}
\end{subfigure}
\begin{subfigure}{0.3\textwidth}
\begin{tikzpicture}[scale=2]
\fill[gray!30!white] (0,0) circle (1cm);
\draw (0,0) circle (1cm);

\fill[red!40!white] (0.2,0.7) circle (0.2cm);
\draw (0.2,0.7) circle (0.2cm);
\draw[blue]  (0.2,0.7) node {$0$};

\fill[red!40!white] (-0.1,0) circle (0.15cm);
\draw (-0.1,0) circle (0.15cm);
\draw[blue] (-0.1,0) node {$0$};

\fill[red!40!white] (0.5,0.2) circle (0.2cm);
\draw (0.5,0.2) circle (0.2cm);
\draw[blue] (0.5,0.2) node {$0$};

\fill[red!40!white] (0.7,-0.13) circle (0.15cm);
\draw (0.7,-0.13) circle (0.15cm);
\draw[blue] (0.7,-0.13) node {$0$};

\fill[red!40!white] (0.4,-0.4) circle (0.15cm);
\draw (0.4,-0.4) circle (0.15cm);
\draw[blue] (0.4,-0.4) node {$0$};

\fill[red!40!white] (0,-0.7) circle (0.12cm);
\draw (0,-0.7) circle (0.12cm);
\draw[blue] (0,-0.7) node {$0$};

\fill[red!40!white] (-0.3,-0.3) circle (0.12cm);
\draw (-0.3,-0.3) circle (0.12cm);
\draw[blue] (-0.3,-0.3) node {$0$};

\fill[red!40!white] (-0.8,-0.05) circle (0.12cm);
\draw (-0.8,-0.05) circle (0.12cm);
\draw[blue] (-0.8,-0.05) node {$0$};

\fill[red!40!white] (-0.6,0.2) circle (0.15cm);
\draw (-0.6,0.2) circle (0.15cm);
\draw[blue] (-0.6,0.2) node {$0$};

\fill[red!40!white] (-0.3,0.4) circle (0.15cm);
\draw (-0.3,0.4) circle (0.15cm);
\draw[blue] (-0.3,0.4) node {$0$};
\end{tikzpicture}
\end{subfigure}
\begin{subfigure}{0.3\textwidth}
\begin{tikzpicture}[scale=2]
\fill[gray!30!white] (0,0) circle (1cm);
\draw (0,0) circle (1cm);

\fill[red!40!white] (0.2,0.3) circle (0.2cm);
\draw (0.2,0.3) circle (0.2cm);
\draw[blue] (0.2,0.3) node {$0$};

\fill[red!40!white] (-0.1,-0.5) circle (0.15cm);
\draw (-0.1,-0.5) circle (0.15cm);
\draw[blue] (-0.1,-0.5) node {$0$};

\fill[red!40!white] (0.6,0.1) circle (0.12cm);
\draw (0.6,0.1) circle (0.12cm);
\draw[blue] (0.6,0.1) node {$0$};
\end{tikzpicture}
\end{subfigure}
\caption{Prototypical examples of balls in the sets $A^{\varepsilon}_1(s^{\varepsilon}_2)$ (left), $A^{\varepsilon}_2(s^{\varepsilon}_2)$ (middle) and $A^{\varepsilon}_3(s^{\varepsilon}_2)$ (right). The balls $B_i^{\varepsilon}(s_2^{\varepsilon})$ are drawn in gray, the corresponding balls in $P_i^{s_2^{\varepsilon}}(0)$ are drawn in red, and their (accumulated) Burgers vectors are drawn in blue, a blue $0$ indicates that the $\mu_{\varepsilon}$-mass of this ball is $0$ .}
\label{fig: Aeps}
\end{figure}
Let us make clear that the sets $P_i^{s_2^{\varepsilon}}(0)$ are meant with respect to the ball construction introduced in the beginning of step 2. \\
Note that a ball can only be in $A^{\varepsilon}_1(s^{\varepsilon}_2)$ if it includes one of the balls from step 1 with non-zero $\mu_{\varepsilon}$-mass. 
The number of these balls was controlled in step 1.
As the balls in $A^{\varepsilon}_1(s^{\varepsilon}_2)$ are by definition of the ball construction disjoint, it follows $\#A^{\varepsilon}_1(s^{\varepsilon}_2) \leq C(\alpha,K,c) |\log \varepsilon|^{1-\delta}$.
In addition, we can argue as in step 1 for the set $\mathcal{M}_{\varepsilon}(s_1^{\varepsilon})$ to obtain that $\# A^{\varepsilon}_2(s^{\varepsilon}_2) \leq C(\alpha,K,c) |\log \varepsilon|^{1-\delta}$. \\
We cannot control the number of balls in $A^{\varepsilon}_3(s^{\varepsilon}_2)$. 
Instead, we will construct a new strain with only slightly more elastic energy and no singularities inside the balls of $A^{\varepsilon}_3(s^{\varepsilon}_2)$ by replacing $\beta_{\varepsilon}$ by local gradients inside these balls.
A similar construction has already been used in \cite{dLGaPo12}(also to delete dipoles) and \cite{MuScZe14}(to extend strains into the cores). \\
Let us pick a ball $B^{\varepsilon}_i(s^{\varepsilon}_2) \in A^{\varepsilon}_3(s^{\varepsilon}_2), i \in I_{\varepsilon}(s^{\varepsilon}_2)$. 
By definition of the set $A^{\varepsilon}_3(s^{\varepsilon}_2)$, there exist natural numbers $0\leq n_1 < \dots < n_{L_i} \leq s^{\varepsilon}_2 - 1$, where $L_i \geq \lfloor \frac{s_2^{\varepsilon}}2 \rfloor$, such that for every $k=1,\dots,L_i$ every $B^{\varepsilon}_j(n_k) \in P^{s^{\varepsilon}_2}_i(n_k)$ does not merge in $(n_k,n_k+1]$.
Note that the annuli $B^{\varepsilon}_j(n_k+1) \setminus B^{\varepsilon}_j(n_k)$, where $B^{\varepsilon}_j(n_k) \in P_i^{s^{\varepsilon}_2}(n_k)$, are pairwise disjoint and contained in $B^{\varepsilon}_i(s^{\varepsilon}_2) \setminus \bigcup_{B \in P_i^{s_2^{\varepsilon}}(0) }B$.
Consequently, it holds
\begin{equation*}
\sum_{k=1}^{L_i} \sum_{B^{\varepsilon}_j(n_k) \in P_i^{s^{\varepsilon}_2}(n_k)} \int_{B^{\varepsilon}_j(n_k +1) \setminus B^{\varepsilon}_j(n_k)} \mathcal{C} \beta_{\varepsilon} : \beta_{\varepsilon} \, dx \leq \int_{B_i(s^{\varepsilon}_2) \setminus \bigcup_{B \in P_i^{s_2^{\varepsilon}}(0) }B} \mathcal{C} \beta_{\varepsilon} : \beta_{\varepsilon} \, dx
\end{equation*}
By the mean value theorem, we may choose $k_i \in \mathbb{N}$ such that 
\begin{align*}
\sum_{B^{\varepsilon}_j(n_{k_i}) \in P_i^{s^{\varepsilon}_2}(n_{k_i})} \int_{B^{\varepsilon}_j(n_{k_i} +1) \setminus B^{\varepsilon}_j(n_{k_i})} \mathcal{C} \beta_{\varepsilon} : \beta_{\varepsilon} \, dx &\leq \frac{1}{L_i} |\log \varepsilon|^2 \, F_{\varepsilon}(\mu_{\varepsilon},\beta_{\varepsilon},B^{\varepsilon}_i(s^{\varepsilon}_2)) \, dx \\
& \leq \frac4{s^{\varepsilon}_2} |\log \varepsilon|^2 \, F_{\varepsilon}(\mu_{\varepsilon},\beta_{\varepsilon},B^{\varepsilon}_i(s^{\varepsilon}_2)),
\end{align*}
where the last inequality holds for $\varepsilon > 0$ small enough. \\
Now, fix $B^{\varepsilon}_j(n_{k_i})  \in P_i^{s^{\varepsilon}_2}(n_{k_i})$. 
By construction, we have $\operatorname{curl }\beta_{\varepsilon} = 0$ in $B^{\varepsilon}_j(n_{k_i}+1) \setminus B^{\varepsilon}_j(n_{k_i}) =: C_{i,j}^{\varepsilon}$.
Moreover, notice that by definition of $A^{\varepsilon}_3(s_2)$ it holds that $\mu_{\varepsilon}(B^{\varepsilon}_j(n_{k_i})) = 0$ (as the ball $B^{\varepsilon}_j(n_{k_i})$ evolves from balls with this property) and therefore 
\begin{equation*}
\int_{\partial B^{\varepsilon}_j(n_{k_i})} \beta_{\varepsilon} \cdot \tau \,d\mathcal{H}^{1} = 0,
\end{equation*}
where $\tau$ denotes the unit tangent to $\partial B^{\varepsilon}_j(n_{k_i})$. 
By standard theory there exists $u_{i,j}^{\varepsilon} \in H^1(C_{i,j}^{\varepsilon};\mathbb{R}^2)$ such that $\beta_{\varepsilon} = \nabla u_{i,j}^{\varepsilon}$ on $C_{i,j}^{\varepsilon}$.
Korn's inequality for the annulus applied to $C_{i,j}^{\varepsilon}$ guarantees the existence of a skew-symmetric matrix $W_{i,j}^{\varepsilon} \in Skew(2)$ such that
\begin{equation*}
\int_{C_{i,j}^{\varepsilon}} |\nabla u_{i,j}^{\varepsilon} - W_{i,j}^{\varepsilon}|^2 \leq K(c) \int_{C_{i,j}^{\varepsilon}} |(\nabla u_{i,j}^{\varepsilon})_{sym}|^2 \,dx.
\end{equation*}
Note that Korn's constant on the right hand side depends only on the ratio of the radii of the annuli $C_{i,j}^{\varepsilon}$ which equals $c$ by construction.
In particular, this constant is independent from $\varepsilon$.  \\
By standard extension results for Sobolev functions there exists a function $v_{i,j}^{\varepsilon} \in H^1(B^{\varepsilon}_j(n_{k_i}+1);\mathbb{R}^2)$ such that $\nabla v_{i,j}^{\varepsilon} = \nabla u_{i,j}^{\varepsilon} - W_{i,j}^{\varepsilon}$ on $C_{i,j}^{\varepsilon}$ and 
\begin{equation*}
\int_{B^{\varepsilon}_j(n_{k_i}+1)} |\nabla v^{\varepsilon}_{i,j}|^2 \, dx \leq C(c) \int_{C_{i,j}^{\varepsilon}} |\nabla u^{\varepsilon}_{i,j} - W^{\varepsilon}_{i,j}|^2 \,dx.
\end{equation*}
Note that by scaling the constant for the extension depends again only on the ratio of the annulus $C_{i,j}^{\varepsilon}$. \\
Now, we can estimate the elastic energy of $\nabla v_{i,j}^{\varepsilon}$ on $B^{\varepsilon}_i(s_2^{\varepsilon})$ by combining the previous two estimates and summing over all balls in $P^{s^{\varepsilon}_2}_i(n_{k_i})$:
\begin{align}
\sum_{B^{\varepsilon}_j(n_{k_i}) \in P^{s^{\varepsilon}_2}_i(n_{k_i})}\int_{B^{\varepsilon}_j(n_{k_i})} \mathcal{C} \nabla v_{i,j}^{\varepsilon} : \nabla v_{i,j}^{\varepsilon} \,dx &\leq C \sum_{B^{\varepsilon}_j(n_{k_i}) \in P^{s^{\varepsilon}_2}_i(n_{k_i})} \int_{B^{\varepsilon}_j(n_{k_i})} |\nabla v_{i,j}^{\varepsilon}i|^2 \,dx  \label{eq: estv1} \\
&\leq C(c) \sum_{B^{\varepsilon}_j(n_{k_i}) \in P^{s^{\varepsilon}_2}_i(n_{k_i})} \int_{C_{i,j}^{\varepsilon}} |\nabla u_{i,j}^{\varepsilon} - W_{i,j}^{\varepsilon}|^2 \, dx \nonumber \\
&\leq C(c) \sum_{B^{\varepsilon}_j(n_{k_i}) \in P^{s^{\varepsilon}_2}_i(n_{k_i})} \int_{C_{i,j}^{\varepsilon}} |(\nabla u_{i,j}^{\varepsilon})_{sym}|^2 \, dx \nonumber \\ 
&\leq C(c) \sum_{B^{\varepsilon}_j(n_{k_i}) \in P^{s^{\varepsilon}_2}_i(n_{k_i})} \int_{C_{i,j}^{\varepsilon}} \mathcal{C} \beta_{\varepsilon} : \beta_{\varepsilon} \,dx \nonumber \\
&\leq C(c) \frac4{s^{\varepsilon}_2} |\log \varepsilon|^2 F_{\varepsilon}(\mu_{\varepsilon},\beta_{\varepsilon},B^{\varepsilon}_i(s^{\varepsilon}_2))  \nonumber \\
&\leq C(c) \frac{8}{\sigma} \frac{\log(c) \, F_{\varepsilon}(\mu_{\varepsilon},\beta_{\varepsilon},B^{\varepsilon}_i(s^{\varepsilon}_2))}{|\log \varepsilon|} |\log \varepsilon|^2, \label{eq: estv2}
\end{align}
where the constant $C(c)$ may change from line to line but depends only on $c$ and global parameters (such as the coercivity constant for $\mathcal{C}$ on symmetric matrices). \\
Let us define the function $\tilde{\beta}_{\varepsilon}: A_{\varepsilon} \rightarrow \mathbb{R}^{2\times2}$ by
\begin{equation*}
\tilde{\beta}_{\varepsilon}(x) = \begin{cases} \nabla v_{i,j}^{\varepsilon}(x) + W_{i,j}^{\varepsilon} &\text{ if } x \in B^{\varepsilon}_j(n_{k_i}) \in P_i^{s^{\varepsilon}_2}(n_{k_i}) \text{ for } B^{\varepsilon}_i(s^{\varepsilon}_2) \in A^{\varepsilon}_3(s^{\varepsilon}_2), \\ \beta_{\varepsilon}(x) &\text{ else.}
\end{cases}
\end{equation*}
Note that on the annuli $C_{i,j}^{\varepsilon}$ it holds $\nabla v_{i,j}^{\varepsilon} + W_{i,j}^{\varepsilon} = \beta_{\varepsilon}$.
Hence, $\tilde{\beta}$ does not create any extra $\operatorname{curl}$ on $\partial B_j^{\varepsilon}(n_{k_i})$. 
Therefore, $\operatorname{curl } \tilde{\beta}_{\varepsilon} = 0$ on $A_{\varepsilon} \setminus \bigcup_{B \in A^{\varepsilon}_1(s^{\varepsilon}_2) \cup A^{\varepsilon}_2(s^{\varepsilon}_2)}B$ where $A^{\varepsilon}_1(s^{\varepsilon}_2) \cup A^{\varepsilon}_2(s^{\varepsilon}_2)$ consists of disjoint balls with a radius less than $\varepsilon^{\alpha+\sigma}$ and $\#(A^{\varepsilon}_1(s^{\varepsilon}_2) \cup A^{\varepsilon}_2(s^{\varepsilon}_2)) \leq C(\alpha,K,c) |\log \varepsilon|^{1-\delta}$.
In particular, the strain $\tilde{\beta}_{\varepsilon}$ satisfies for every open $A \subseteq A_{\varepsilon} \setminus \bigcup_{B \in A^{\varepsilon}_1(s^{\varepsilon}_2) \cup A^{\varepsilon}_2(s^{\varepsilon}_2)} B$ with smooth boundary the circulation condition 
\begin{equation*}
\int_{\partial A} \tilde{\beta}_{\varepsilon} \cdot \tau \,d\mathcal{H}^1 = \tilde{\mu}_{\varepsilon}(A),
\end{equation*}
where $\tau$ denotes the unit tangent to $\partial A$ and $\tilde{\mu}_{\varepsilon}=(\mu_{\varepsilon})_{|\bigcup_{B \in A^{\varepsilon}_1(s^{\varepsilon}_2) \cup A^{\varepsilon}_2(s^{\varepsilon}_2)} B}$. 
Note that $\tilde{\mu}_{\varepsilon}(U) = \mu_{\varepsilon}(U)$ for any connected component $U$ of $A_{\varepsilon}$ as we only deleted connected dipoles.\\
Moreover, in view of \eqref{eq: estv1} -- \eqref{eq: estv2} it holds
\begin{equation} \label{eq: estbtilde}
\frac1{|\log \varepsilon|^2}\int_{A_{\varepsilon} \setminus \left(\bigcup_{B \in A^{\varepsilon}_1(s^{\varepsilon}_2) \cup A^{\varepsilon}_2(s^{\varepsilon}_2)} B\right)} \mathcal{C} \tilde{\beta}_{\varepsilon} : \tilde{\beta}_{\varepsilon} \,dx \leq F_{\varepsilon}(\mu_{\varepsilon},\beta_{\varepsilon},A_{\varepsilon}) + C(c,\alpha) \frac{F_{\varepsilon}(\mu_{\varepsilon},\beta_{\varepsilon},A_{\varepsilon})}{|\log \varepsilon|}.
\end{equation} 
Eventually, note that $\tilde{\beta}_{\varepsilon} = \beta_{\varepsilon}$ outside the balls in $A^{\varepsilon}_3(s^{\varepsilon}_2)$, which are all included in $\bigcup_{x \in \operatorname{supp }(\mu_{\varepsilon})} B_{\varepsilon^{\alpha}}(x)$.
\newline \\
\textbf{Step 3. }\emph{Replacing the circulation condition by a measure-valued $\operatorname{curl}$.} \\
We know from Step 2 that $\#(A^{\varepsilon}_1(s^{\varepsilon}_2) \cup A^{\varepsilon}_2(s^{\varepsilon}_2)) \leq C(\alpha,K,c) |\log \varepsilon|^{1-\delta}$. 
Now, choose $c_1 = c_1(\alpha,K,c) > 1$ such that $\log c_1 = \frac18 \frac{\sigma}{C(\alpha,K,c)}$ where $c>1$ is the universal expanding factor of the ball constructions in step 1 and 2. \\
Consider a ball construction associated to $c_1$ starting with the balls in $A^{\varepsilon}_1(s^{\varepsilon}_2) \cup A^{\varepsilon}_2(s^{\varepsilon}_3)$.
Again, denote its output by $(I_{\varepsilon}(t), (B_i^{\varepsilon}(t))_{i \in I_{\varepsilon}(t)}, (R_i^{\varepsilon})_{i \in I_{\varepsilon}(t)})_t$. \\
From step 2 we know that for every ball $B \in  A^{\varepsilon}_1(s^{\varepsilon}_2) \cup A^{\varepsilon}_2(s^{\varepsilon}_2)$ it holds $\operatorname{diam} B \leq 2 \varepsilon^{\alpha + \sigma}$.
Arguing as in step 1 and 2, we obtain that for $\varepsilon >0$ small enough it holds that for all $t \leq \lceil \frac{\sigma}2 \frac{|\log \varepsilon|}{\log c_1} \rceil =:s^{\varepsilon}_3$ we have $\sum_{i \in I_{\varepsilon}(t)} R_i(t) \leq \varepsilon^{\alpha}$.
During the construction, the number of merging times is definitely bounded by the number of starting balls i.e., less than $C(\alpha,K,c) |\log \varepsilon|^{1-\delta}$.
Hence, there are at least $s^{\varepsilon}_3 - C(\alpha,K,c) |\log \varepsilon|^{1-\delta}$ natural numbers $n \leq s_3^{\varepsilon}-1$ such that there is no merging time in the interval $(n,n+1]$.
By the definition of $c_1$ we observe for $\varepsilon>0$ small enough that
\begin{equation}\label{eq: estseps3}
s^{\varepsilon}_3 - C(\alpha,K,k,c) |\log \varepsilon|^{1-\delta} \geq \frac23 s_3^{\varepsilon}.
\end{equation}
In particular, there exist natural numbers $\frac{s_3^{\varepsilon}}2 \leq n_1 < \dots < n_L \leq s_3^{\varepsilon} - 1$, $L \geq \frac{s_3^{\varepsilon}}7$, such that none of the balls $(B^{\varepsilon}_i(n_k))_{i \in I_{\varepsilon}(n_k)}$ merges in the time interval $(n_k,n_k+1]$.
As in step 2, by the mean value theorem, we can find a natural number $\frac{s_3^{\varepsilon}}2 \leq n_k, 1 \leq k \leq L$, which satisfies in addition
\begin{equation*}
\sum_{i \in I_{\varepsilon}(n_k)} \int_{B^{\varepsilon}_i(n_k + 1) \setminus B^{\varepsilon}_i(n_k)} \mathcal{C} \tilde{\beta}_{\varepsilon} : \tilde{\beta}_{\varepsilon} \, dx \leq \frac7{s^{\varepsilon}_3} \int_{A_{\varepsilon} \setminus \left(\bigcup_{B \in A^{\varepsilon}_1(s_2^{\varepsilon}) \cup A^{\varepsilon}_2(s_2^{\varepsilon})} B\right)} \mathcal{C} \tilde{\beta}_{\varepsilon} : \tilde{\beta}_{\varepsilon} \, dx. 
\end{equation*} 
For $i \in I_{\varepsilon}(n_k)$ we perform the following construction.
Let $\xi_i = \tilde{\mu}_{\varepsilon}(B^{\varepsilon}_i(n_k))$, where $\tilde{\mu}_{\varepsilon}$ is defined as in the end of  step 2, and define the function 
\begin{equation*}
K_i(x) = \frac1{2\pi} \frac{\xi_i \otimes J (x-x^{\varepsilon}_i)}{|x-x^{\varepsilon}_i|^2}.
\end{equation*}
Here, $x^{\varepsilon}_i$ is the center of the ball $B^{\varepsilon}_i(n_k)$ and $J$ is the clockwise rotation by $\frac{\pi}2$.
A straightforward computation shows that $\operatorname{curl }K_i = 0$ on $B^{\varepsilon}_i(n_k + 1) \setminus B^{\varepsilon}_i(n_k) =: C^{\varepsilon}_i(n_k)$  and
\begin{equation*}
\int_{C^{\varepsilon}_i(n_k)} |K_i|^2 \,dx =  |\xi_i|^2 \frac{\log(c_1)}{2\pi} \leq C(c_1) \int_{C^{\varepsilon}_i(n_k)} \mathcal{C} \tilde{\beta}_{\varepsilon} : \tilde{\beta}_{\varepsilon} \, dx. 
\end{equation*}
For the inequality, we used Lemma \ref{lemma: boundcirc}. \\
Moreover, we notice that 
\begin{equation*}
\operatorname{curl }(\tilde{\beta}_{\varepsilon} - K_i) = 0 \text{ in } C^{\varepsilon}_i(n_k) \text{ and } \int_{\partial B^{\varepsilon}_i(n_k)} (\tilde{\beta}_{\varepsilon} - K_i) \cdot \tau \,dx = 0.
\end{equation*}
Consequently, there exists a function $u^{\varepsilon}_i \in H^1(C^{\varepsilon}_i(n_k);\mathbb{R}^2)$ such that $\nabla u^{\varepsilon}_i = \tilde{\beta}_{\varepsilon} - K_i$ on $C^{\varepsilon}_i(n_k)$.
Similar to step 2, we can apply Korn\rq{}s inequality for the annulus on $C^{\varepsilon}_i(n_k)$ to obtain a skew-symmetric matrix $W^{\varepsilon}_i \in Skew(2)$ such that
\begin{equation*}
\int_{C^{\varepsilon}_i(n_k)} |\nabla u^{\varepsilon}_i - W^{\varepsilon}_i|^2 \, dx \leq C(c_1) \int_{C^{\varepsilon}_i(n_k)} |(\nabla u^{\varepsilon}_i)_{sym}|^2 \, dx.
\end{equation*}
Note again that the constant depends only on the ratio of the annulus $C^{\varepsilon}_i(n_k)$ which is by construction $c_1$. \\
In addition, by classical extension results, there exists a function $v^{\varepsilon}_i \in H^1(B^{\varepsilon}_i(n_k+1);\mathbb{R}^2)$ such that $\nabla v^{\varepsilon}_i = \nabla u_i - W^{\varepsilon}_i$ and 
\begin{equation*}
\int_{B^{\varepsilon}_i(n_k+1)} |\nabla v^{\varepsilon}_i|^2 \leq C(c_1) \int_{C^{\varepsilon}_i(n_k)} |\nabla u^{\varepsilon}_i - W_i^{\varepsilon}|^2 \,dx.
\end{equation*}
By scaling, also the constant on the right hand side of this inequality depends only on $c_1$. \\
Combining the last four estimates and summing over $i \in I_{\varepsilon}(n_k)$ yields the following chain of inequalities
\begin{align}
\sum_{i \in I_{\varepsilon}(n_k)} \int_{B^{\varepsilon}_i(n_k + 1)} \mathcal{C} \nabla v^{\varepsilon}_i : \nabla v^{\varepsilon}_i \,dx &\leq C \sum_{i \in I_{\varepsilon}(n_k)} \int_{B^{\varepsilon}_i(n_k + 1)} |\nabla v^{\varepsilon}_i|^2 \,dx \label{eq: barbeta1}\\
&\leq C(c_1) \sum_{i \in I_{\varepsilon}(n_k)} \int_{C^{\varepsilon}_i(n_k)} |\nabla u^{\varepsilon}_i - W_i|^2 \,dx \nonumber\\
&\leq C(c_1) \sum_{i \in I_{\varepsilon}(n_k)} \int_{C^{\varepsilon}_i(n_k)} |(\nabla u^{\varepsilon}_i)_{sym}|^2 \,dx \nonumber\\
&=  C(c_1) \sum_{i \in I_{\varepsilon}(n_k)} \int_{C^{\varepsilon}_i(n_k)} |(\tilde{\beta}_{\varepsilon} - K_i)_{sym}|^2 \,dx \nonumber\\
&\leq C(c_1) \sum_{i \in I_{\varepsilon}(n_k)} \int_{C^{\varepsilon}_i(n_k)} \mathcal{C} \tilde{\beta}_{\varepsilon} : \tilde{\beta}_{\varepsilon} \,dx + \int_{C^{\varepsilon}_i(n_k)} |K_i|^2 \, dx \nonumber\\
&\leq C(c_1) \sum_{i \in I_{\varepsilon}(n_k)} \int_{C^{\varepsilon}_i(n_k)} \mathcal{C} \tilde{\beta}_{\varepsilon} : \tilde{\beta}_{\varepsilon} \,dx \nonumber \\
&\leq C(c_1) \frac7{s^{\varepsilon}_3} \int_{A_{\varepsilon}\setminus \left(\bigcup_{B \in A^{\varepsilon}_1(s^{\varepsilon}_2) \cup A^{\varepsilon}_2(s^{\varepsilon}_2)}B\right)} \mathcal{C} \tilde{\beta}_{\varepsilon} : \tilde{\beta}_{\varepsilon} \,dx \nonumber \\
&\leq \frac{C(\alpha,c_1) }{|\log \varepsilon|} \int_{A_{\varepsilon}\setminus \left( \bigcup_{B \in A^{\varepsilon}_1(s^{\varepsilon}_2) \cup A^{\varepsilon}_2(s^{\varepsilon}_2)}B\right)} \mathcal{C} \tilde{\beta}_{\varepsilon} : \tilde{\beta}_{\varepsilon} \,dx. \label{eq: barbeta2}
\end{align}
Here, the constant $C(c_1)$ changed from line to line but it depends only on $c_1$ and global parameters.
Now, define the strain $\bar{\beta}_{\varepsilon}: A_{\varepsilon} \rightarrow \mathbb{R}^{2\times2}$ by
\begin{equation*}
\bar{\beta}_{\varepsilon}(x) = \begin{cases} \nabla v^{\varepsilon}_i(x) + W_i &\text{ if } x \in B^{\varepsilon}_i(n_k+1) \text{ for some } i \in I_{\varepsilon}(n_k), \\
\tilde{\beta}_{\varepsilon}(x) &\text{ else}.
\end{cases}
\end{equation*}
Note that as the balls $(B^{\varepsilon}_i(n_k+1))_{i \in I_{\varepsilon}(n_k +1 )}$ are disjoint, $\bar{\beta}_{\varepsilon}$ is well-defined.
Moreover, from \eqref{eq: barbeta1} -- \eqref{eq: barbeta2} and \eqref{eq: estbtilde} in step 2 we derive that
\begin{align*}
\frac1{|\log \varepsilon|^2}\int_{A_{\varepsilon}} \frac12 \mathcal{C} \bar{\beta}_{\varepsilon}:\bar{\beta}_{\varepsilon} \, dx 
&\leq \frac1{|\log \varepsilon|^2} \left(1 + \frac{C(\alpha,c_1)}{|\log \varepsilon|}\right) \int_{A_{\varepsilon}\setminus \bigcup_{B \in A^{\varepsilon}_1(s^{\varepsilon}_2) \cup A^{\varepsilon}_2(s^{\varepsilon}_2)}B} \frac12 \mathcal{C} \tilde{\beta}_{\varepsilon} : \tilde{\beta}_{\varepsilon} \,dx \\
&\leq \left(1 + \frac{C(\alpha,c)}{|\log \varepsilon|}\right)\left(1+\frac{C(\alpha,c_1)}{|\log \varepsilon|}\right) F_{\varepsilon}(\mu_{\varepsilon},\beta_{\varepsilon},A_{\varepsilon}).
\end{align*}
In addition, it holds $\operatorname{curl }\bar{\beta}_{\varepsilon} = \sum_{i \in I_{\varepsilon}(n_k)} (K_i \cdot \tau) \,\mathcal{H}^1_{|\partial B^{\varepsilon}_i(n_k +1)}$ where $\tau$ is the unit tangent to $\partial B^{\varepsilon}_i(n_k +1)$,
\begin{equation*}
|\operatorname{curl }\bar{\beta}_{\varepsilon}| = \sum_{i \in I_{\varepsilon}(n_k)} |K_i| \,\mathcal{H}^1_{|\partial B^{\varepsilon}_i(n_k +1)}, \text{ and } \int_{\partial B^{\varepsilon}_i(n_k +1)} |K_i| \,d\mathcal{H}^1 = |\xi_i| = |\tilde{\mu}_{\varepsilon} (B^{\varepsilon}_i(n_k +1))|.
\end{equation*}
Finally, set $I_{\varepsilon} = I_{\varepsilon}(n_k+1)$ and $(D_i^{\varepsilon})_{i\in I_{\varepsilon}} = (B^{\varepsilon}_i(n_k+1))_{i \in I_{\varepsilon}(n_k+1)}$.
Then \ref{item: lowerbound11}, \ref{item: lowerbound12}, \ref{item: lowerbound 13}, \ref{item: lowerbound14}, and \ref{item: lowerbound1last} are fulfilled. 
As also in the third step we changed the function from step 2 only in $\bigcup_{x \in \operatorname{supp }(\mu_{\varepsilon})} B_{\varepsilon^{\alpha}}(x)$, it follows \eqref{lowerbound: lowerbound12ab}. \\
Hence, it is left to show \ref{item: lowerbound15}.
Recall that $n_k \geq \frac{s^{\varepsilon}_3}{2}$.
By \eqref{eq: estseps3}, there exist at least $\frac{s^{\varepsilon}_3}7$ natural numbers $n$ below $n_k-1$ such that there is no merging time between $n$ and $n+1$.
A similar computation to  \eqref{eq: lowerboundstep11} - \eqref{eq: lowerboundstep1last} in step 1 shows that
\begin{equation*}
|\operatorname{curl} \bar{\beta}_{\varepsilon}|(A_{\varepsilon}) = \sum_{i \in I_{\varepsilon}} |\tilde{\mu}_{\varepsilon}(D_i^{\varepsilon})| \leq C(\alpha,K,c_1) |\log \varepsilon|^{1 - \delta},
\end{equation*}
which is \ref{item: lowerbound15}. \\
Eventually, note that $c_1$ depends only on $\alpha,K$, and $c$ where $c$ is a fixed universal parameter.
\end{proof}

The Proposition above allows us to reduce the complicated situation with at most $|\log \varepsilon|^2$ dislocations to a simpler one.
After applying the previous proposition, there are only $\sim|\log \varepsilon|$ balls in which the $\operatorname{curl}$ of the modified strain is concentrated. 
This will be enough to obtain compactness.
For the $\liminf$-inequality, one needs to compute self-energies of dislocations.
The self-energy density $\psi$ as defined in \eqref{def: selfenergy} is the renormalized limit of energies computed for $\operatorname{curl}$-free functions on annuli whose ratios go to infinity.
As we want to derive the same quantities also in this situation, it is necessary that we are able to find annuli around the dislocation cores with growing ratios in which the strain (respectively the modified strain in the sense of the previous proposition) is $\operatorname{curl}$-free.
The previous proposition for $\delta=0$ guarantees essentially only the existence of annuli with a fixed ratio uniformly in $\varepsilon$. \\ 
The next proposition shows that either most of the dislocations allow growing ratios in a ball construction or the accumulated Burgers vector is small and the previous proposition allows to reduce the situation to less than $|\log \varepsilon|^{1-\delta}$ dislocation balls for $\delta >0$.
The latter case leads in average to growing differences between consecutive merging times in a ball construction which
 is equivalent to growing ratios of the corresponding annuli. 

\begin{proposition}\label{prop: lowerbound2}
Let $1 > \alpha > \gamma > 0, \, K > 0, l > 0$, $c>1$ and $\frac15 > \delta > 0$.
Then there exists $\varepsilon_0 = \varepsilon_0(\alpha,\gamma,\delta,K,l,c) > 0$ such that for all $0 < \varepsilon < \varepsilon_0$ the following holds: \newline \\
Let $A_{\varepsilon} \subseteq \mathbb{R}^2$. 
Let $(B_i^{\varepsilon})_{i \in I_{\varepsilon}}$ be a family of disjoint balls in $A_{\varepsilon}$ such that
\begin{itemize}
\item[$\bullet$] $\operatorname{diam} B_i^{\varepsilon} \leq \varepsilon^{\alpha}$ for all $i \in I_{\varepsilon}$,
\item[$\bullet$] $|I_{\varepsilon}| \leq K |\log \varepsilon|$,
\item[$\bullet$] $\operatorname{dist}(B_i^{\varepsilon},\partial A_{\varepsilon}) \geq l \varepsilon^{\gamma}$.
\end{itemize}
Let $\beta_{\varepsilon}: A_{\varepsilon} \rightarrow \mathbb{R}^{2\times2}$ such that $\tilde{\mu}_{\varepsilon} = \operatorname{curl}(\beta_{\varepsilon}) \in \mathcal{M}(A_{\varepsilon};\R^2)$ and $\operatorname{supp}(\operatorname{curl} \beta_{\varepsilon}) \subseteq \bigcup_{i \in I_{\varepsilon}} B_i^{\varepsilon}$.
Moreover, assume that $\int_{A_{\varepsilon}} \mathcal{C} \beta_{\varepsilon}:\beta_{\varepsilon} \,dx + |\operatorname{curl }\beta_{\varepsilon}|(A_{\varepsilon})^2 \leq K |\log \varepsilon|^2$.
Then at least one of the following options holds true:
\begin{enumerate}
\item $|(\operatorname{curl }\beta_{\varepsilon})(A_{\varepsilon})| \leq |\log \varepsilon|^{1-\delta}$,
\item Consider a ball construction associated to $c$ starting with the balls $(B_i^{\varepsilon})_{i \in I_{\varepsilon}}$ and the time $t_s^{\varepsilon}$ which is defined to be the first time such that a ball in the ball construction intersects $\partial A_{\varepsilon}$. Then there exists a subset $\tilde{I}_{\varepsilon} \subseteq I_{\varepsilon}(t_s^{\varepsilon})$ such that for any $i \in \tilde{I}_{\varepsilon}$ there exist at most $|\log \varepsilon|^{\delta}$-many times $n \in \mathbb{N}$, $n \leq t_s^{\varepsilon}-1$, such that at least one ball in $P_i^{t_s^{\varepsilon}}(n)$ merges in the time interval $(n,n+1]$. Moreover, it holds
\begin{equation}
\left| \sum_{i \in \tilde{I}_{\varepsilon}} \tilde{\mu}_{\varepsilon}(B_i^{\varepsilon}) - \tilde{\mu}_{\varepsilon}(A_{\varepsilon}) \right| \leq \delta |\tilde{\mu}_{\varepsilon}(A_{\varepsilon})|.
\end{equation}
\label{eq: pauli}
\end{enumerate}
\end{proposition}
\begin{proof}
Let $\frac15 > \delta > 0$. 
Fix $c>1$. 
Let us perform a ball construction associated to $c$ starting with the balls $(B_{i}^{\varepsilon})_{i \in I_{\varepsilon}}$.
As in the previous proof, we denote the output of the construction by $(I_{\varepsilon}(t) , (B_i^{\varepsilon}(t))_{i \in I_{\varepsilon}(t)}, (R_i^{\varepsilon}(t))_{i \in I_{\varepsilon}(t)})_t$. \\
Let $t_s^{\varepsilon}$ be the first time at which one of the balls in $(B_i^{\varepsilon}(t_s^{\varepsilon}))_{i \in I_{\varepsilon}(t_s^{\varepsilon})}$ intersects $\partial A_{\varepsilon}$.
If $t_s^{\varepsilon}$ is a merging time, still denote by $(B_i^{\varepsilon}(t_s^{\varepsilon}))_{i \in I_{\varepsilon}(t_s^{\varepsilon})}$ the unmerged balls whose pairwise intersection is a set of $\mathcal{L}^2$-measure zero.
As the balls $(B_i^{\varepsilon})_{i \in I_{\varepsilon}}$ have radii not larger than $\varepsilon^{\alpha}$ and a distance of at least $l \varepsilon^{\gamma}$ to $\partial A_{\varepsilon}$, we can argue as in the previous proof to obtain that for $\varepsilon$ small enough it holds that  $t^{\varepsilon}_s \geq \lceil \frac{\alpha - \gamma}{2} \frac{|\log \varepsilon|}{\log c} \rceil \gg |\log \varepsilon|^{1-\delta}$. 
Let us define the set of balls which are affected by at most $|\log \varepsilon|^{1-\delta}$ discrete merging steps by
\begin{align*}
\mathcal{G}_{\varepsilon} =\{ &B^{\varepsilon}_i(t^{\varepsilon}_s): i \in I_{\varepsilon}(t^{\varepsilon}_s) \text{ and there exist more than } t^{\varepsilon}_s - |\log \varepsilon|^{1-\delta} \text{ natural numbers } \\ & 0 \leq n_1 < \dots < n_{L_i}\leq t^{\varepsilon}_s -1 \text{ such that for all } 1 \leq k \leq  L_i \text{ none of the balls } B^{\varepsilon}_j(n_k) \in P_i^{t^{\varepsilon}_s}(n_k) \\ & \text{ merges in the interval } (n_k,n_k+1] \}
\intertext{ and its parents at time $0 < t < t_s^{\varepsilon}$ by} 
\mathcal{G}_{\varepsilon}(t) =\{&B^{\varepsilon}_j(t): j \in I_{\varepsilon}(t) \text{ and there is } i \in I_{\varepsilon}(t^{\varepsilon}_s) \text{ such that } B^{\varepsilon}_j(t) \subseteq B^{\varepsilon}_i(t^{\varepsilon}_s) \in \mathcal{G}_{\varepsilon} \}.
\end{align*}
Analogously we denote the set of balls that are involved in mergings in many discrete steps by $\mathcal{B}_{\varepsilon} = \{B^{\varepsilon}_i(t^{\varepsilon}_s): i \in I_{\varepsilon}(t^{\varepsilon}_s)\} \setminus \mathcal{G}_{\varepsilon}$ and its parents $\mathcal{B}_{\varepsilon}(t)$ at time $t>0$. \\
In the following, we will show that if the balls in $\mathcal{G}_{\varepsilon}$ do not carry most of the mass of $\tilde{\mu}_{\varepsilon}$, then $\tilde{\mu}_{\varepsilon}(A_{\varepsilon})$ has to be small itself. 
\newline\\
\underline{\emph{Claim:}} If $\left| \sum_{B \in \mathcal{G}_{\varepsilon}} \tilde{\mu}_{\varepsilon}(B) - \tilde{\mu}_{\varepsilon}(A_{\varepsilon})  \right| \geq \delta |\tilde{\mu}_{\varepsilon}|(A_{\varepsilon})$, then $|\tilde{\mu}_{\varepsilon}(A_{\varepsilon})| \leq |\log \varepsilon|^{1 - \delta}$ for $\varepsilon$ small enough depending on $c,\delta,\gamma$ and $\alpha$. 
\newline \\
Let $\left| \sum_{B \in \mathcal{G}_{\varepsilon}} \tilde{\mu}_{\varepsilon}(B) - \tilde{\mu}_{\varepsilon}(A_{\varepsilon})  \right| \geq \delta |\tilde{\mu}_{\varepsilon}|(A_{\varepsilon})$ but let us assume that $|\tilde{\mu}(A_{\varepsilon})| > |\log \varepsilon|^{1-\delta}$. \\
First, we apply the generalized Korn inequality (see \cite[Theorem 11]{GaLePo10}) for any $B^{\varepsilon}_i(t^{\varepsilon}_s) \in \mathcal{B}_{\varepsilon}$ i.e., for any ball $B^{\varepsilon}_i(t^{\varepsilon}_s) \in \mathcal{B}_{\varepsilon}$ there exists a skew-symmetric $W^{\varepsilon}_i \in Skew(2)$ such that
\begin{equation*}
\int_{B^{\varepsilon}_i(t^{\varepsilon}_s)} | \beta_{\varepsilon} - W^{\varepsilon}_i |^2 \,dx \leq C \left( \int_{B^{\varepsilon}_i(t^{\varepsilon}_s)} |(\beta_{\varepsilon})_{sym} |^2 \,dx + (|\tilde{\mu}_{\varepsilon}|(B^{\varepsilon}_i(t^{\varepsilon}_s)))^2 \right).
\end{equation*}
Note that by scaling the constant does not depend on the size of the ball. \\
Summing over all $i \in I_{\varepsilon}(t^{\varepsilon}_s)$ such that $B^{\varepsilon}_i(t^{\varepsilon}_s) \in \mathcal{B}_{\varepsilon}$ yields (recall that by construction the pairwise intersections of balls in $(B^{\varepsilon}_i(t^{\varepsilon}_s))_{i \in I_{\varepsilon}(t^{\varepsilon}_s)}$ are of negligible Lebesgue measure)  
\begin{equation} \label{eq: boundgekorn}
\sum_{B^{\varepsilon}_i(t^{\varepsilon}_s) \in \mathcal{B}_{\varepsilon}} \int_{B^{\varepsilon}_i(t^{\varepsilon}_s)} | \beta_{\varepsilon} - W^{\varepsilon}_i |^2 \,dx \leq C \left( \int_{A_{\varepsilon}} \mathcal{C} \beta_{\varepsilon}: \beta_{\varepsilon} \,dx + \sum_{B^{\varepsilon}_i(t^{\varepsilon}_s) \in \mathcal{B}_{\varepsilon}} (|\tilde{\mu}_{\varepsilon}|(B^{\varepsilon}_i(t^{\varepsilon}_s)))^2 \right) \leq  C K |\log \varepsilon|^2.
\end{equation}
For the last inequality, we used the simple estimate (recall that $\tilde{\mu}_{\varepsilon} = \operatorname{curl} \beta_{\varepsilon}$ is concentrated in the family $(B_i^{\varepsilon})_{i \in I_{\varepsilon}}$ which consists of much smaller balls than the ones in $(B^{\varepsilon}_i(t^{\varepsilon}_s))_{i \in I_{\varepsilon}(t^{\varepsilon}_s)}$)
\begin{equation*}
\sum_{B_i(t^{\varepsilon}_s) \in \mathcal{B}_{\varepsilon}} (|\tilde{\mu}_{\varepsilon}|(B^{\varepsilon}_i(t^{\varepsilon}_s)))^2 \leq \left(\sum_{B^{\varepsilon}_i(t^{\varepsilon}_s) \in \mathcal{B}_{\varepsilon}} |\tilde{\mu}_{\varepsilon}|(B^{\varepsilon}_i(t^{\varepsilon}_s))\right)^2 \leq \left( |\tilde{\mu}_{\varepsilon}|(A_{\varepsilon}) \right)^2.
\end{equation*}
In the following, we find a lower bound for the energy concentrated on the balls of $\mathcal{B}_{\varepsilon}$. \\
First, notice that the balls in $\mathcal{B}_{\varepsilon}$ emerge from mergings which are distributed over at least $|\log \varepsilon|^{1-\delta}$ time steps of the form $(n,n+1]$ for some $n \leq t_s^{\varepsilon}-1$.
Arguing as for the sets $\mathcal{M}_{\varepsilon}(s^{\varepsilon}_1)$ in step 1 of the proof of the previous proposition, we  can obtain that for $t_s^{\varepsilon} \geq t \geq t^{\varepsilon}_s - \frac{|\log \varepsilon|^{1-\delta}}{2}$ it holds that 
\begin{equation}
\# \mathcal{B}_{\varepsilon}(t) \leq C \frac{\# I_{\varepsilon}}{|\log \varepsilon|^{1 - \delta}} \leq C K |\log \varepsilon|^{\delta}. \label{eq: estB}
\end{equation}
Let us denote by $\tau^{\varepsilon}_1 < \dots < \tau^{\varepsilon}_{L_{\varepsilon}}$ the merging times between $t^{\varepsilon}_s - \frac{|\log \varepsilon|^{1-\delta}}2$ and $t^{\varepsilon}_s$. 
Moreover, write $\tau^{\varepsilon}_0 = t^{\varepsilon}_s - \frac{|\log \varepsilon|^{1-\delta}}2$ and $\tau^{\varepsilon}_{L_{\varepsilon}+1} = t^{\varepsilon}_s$.
From estimate \eqref{eq: estB}, we derive that that for any $ 0 \leq k \leq L_{\varepsilon}$ there exists $i_k \in I_{\varepsilon}(\tau^{\varepsilon}_k)$ such that 
\begin{equation} \label{eq: boundbball}
B^{\varepsilon}_{i_k}(\tau^{\varepsilon}_k) \in \mathcal{B}_{\varepsilon}(\tau^{\varepsilon}_k) \text{ and } |\tilde{\mu}_{\varepsilon}(B^{\varepsilon}_{i_k}(\tau^{\varepsilon}_k))| \geq \frac{\delta}{2K} |\log \varepsilon|^{1- 2 \delta}.
\end{equation} 
Here, we used that $\left| \sum_{B^{\varepsilon}_{i}(\tau^{\varepsilon}_k) \in \mathcal{B}_{\varepsilon}(\tau^{\varepsilon}_k)} \tilde{\mu}_{\varepsilon}(B^{\varepsilon}_i(\tau^{\varepsilon}_k)) \right| = \left| \sum_{B \in \mathcal{G}_{\varepsilon}} \tilde{\mu}_{\varepsilon}(B) - \tilde{\mu}_{\varepsilon}(A_{\varepsilon}) \right| > \delta |\log \varepsilon|^{1-\delta}$.
By Lemma \ref{lemma: boundcirc} and \eqref{eq: boundbball}, we estimate for $j \in I_{\varepsilon}(t^{\varepsilon}_s)$ such that $B^{\varepsilon}_{i_k}(\tau^{\varepsilon}_k) \subseteq B^{\varepsilon}_j(t^{\varepsilon}_s) \in \mathcal{B}_{\varepsilon}$ the following:
\begin{align*}
\int_{ B^{\varepsilon}_{i_k}(\tau^{\varepsilon}_{k+1}) \setminus B^{\varepsilon}_{i_k}(\tau^{\varepsilon}_k) } |\beta_{\varepsilon} - W^{\varepsilon}_j|^2 \, dx &\geq \frac{1}{2\pi} (\tau^{\varepsilon}_{k+1} - \tau^{\varepsilon}_k) (\log c) \, (|\mu_{\varepsilon}(B^{\varepsilon}_{i_k}(\tau^{\varepsilon}_k))|)^2 \\
&\geq (\tau^{\varepsilon}_{k+1} - \tau^{\varepsilon}_k) (\log c) \, \frac{\delta^2}{8 \pi K^2} |\log \varepsilon|^{2-4 \delta}.
\end{align*}
Summing over all merging times between $t_s^{\varepsilon} - \frac{|\log \varepsilon|^{1-\delta}}2$ and $t_s^{\varepsilon}$ provides the estimate
\begin{equation*}
\sum_{k=0}^{L_{\varepsilon}} \int_{ B^{\varepsilon}_{i_k}(\tau^{\varepsilon}_{k+1}) \setminus B^{\varepsilon}_{i_k}(\tau^{\varepsilon}_k) } |\beta_{\varepsilon} - W^{\varepsilon}_j|^2 \, dx \geq  \frac{\delta^2}{16 \pi K^2} (\log c) \; |\log \varepsilon|^{3-5 \delta}.
\end{equation*}
Together with \eqref{eq: boundgekorn}, this implies
\begin{equation*}
 \frac{\delta^2}{16 \pi K^2} \log(c) \, |\log \varepsilon|^{3-5 \delta} \leq 2 C K |\log \varepsilon|^2,
\end{equation*}
which is a contradiction for $\delta < \frac 15$ and $\varepsilon>0$ small enough depending on $\delta$ and the occurring constants.
Hence, the claim is proven. \\
As the balls in $\mathcal{G}_{\varepsilon}$ have the claimed property \eqref{eq: pauli}, this finishes the proof.
\end{proof}
\begin{remark}
Note that both propositions in this section hold also if one replaces the elastic tensor by a nonlinear energy density with quadratic growth.
In the proof one only needs to replace Korn's inequality by its non-linear counterpart.
\end{remark}
%%%%%%%%%%%%%%%%%%%%%%%%%%%%%%%%
\subsection{Compactness}\label{section: comp}

In this section, we prove the compactness result, Theorem \ref{thm: compactness}.

\begin{proof}
\textbf{Step 1. }\emph{Compactness of the dislocation measures.} \\
Fix $1  > \gamma >0$ and define $\alpha = \gamma + \frac{1-\gamma}{2}$. Then $\gamma < \alpha < 1$.  \\
Denote by $(A_{\varepsilon_k}^j)_{j \in J_{\varepsilon_k}}$ the connected components of $\bigcup_{x \in \operatorname{supp}( \mu_k)} B_{\varepsilon_k^{\gamma}}(x)$ that do not intersect $\partial \Omega$.
Define $U_{\varepsilon_k} = \bigcup_{j \in J_{\varepsilon_k}} A_{\varepsilon_k}^j$. \\
We apply Proposition \ref{prop: lowerbound1} on $U_{\varepsilon_k}$ to $\mu_k$, $\beta_k$, and $\delta = 0$ which provides a strain $\tilde{\beta}_k: U_{\varepsilon_k} \rightarrow \mathbb{R}^{2\times2}$ satisfying
\begin{equation*}
|\operatorname{curl } \tilde{\beta}_k|(U_{\varepsilon_k}) \leq C(\gamma) |\log \varepsilon_k|.
\end{equation*}
Moreover, by \ref{lowerbound: lowerbound12ab} of Proposition \ref{prop: lowerbound1} we can extend $\tilde{\beta}_k$ by $\beta_k$ to $\Omega \setminus  U_{\varepsilon_k}$ without creating additional $\operatorname{curl}$ on $\partial U_{\varepsilon_k}$. 
In the following, we call this extended function also $\tilde{\beta}_k$. \\
Let us write $\tilde{\mu}_k =( \operatorname{curl } \tilde{\beta}_k)_{|U_{\varepsilon_k}} \in \mathcal{M}(\Omega;\R^2)$
which fulfills $\frac{|\tilde{\mu}_k|}{|\log \varepsilon_k|} \leq C(\gamma)$.
Hence, there exist a measure $\mu \in \mathcal{M}(\Omega;\mathbb{R}^2)$ and  a (not relabeled) subsequence such that $\frac{\tilde{\mu}_k}{| \log \varepsilon_k|} \stackrel{*}{\rightharpoonup} \mu$ in $\mathcal{M}(\Omega;\mathbb{R}^2)$.
As the embedding $W^{1,\infty}\hookrightarrow  C_0^0(\Omega)$ is compact, it follows that $\frac{\tilde{\mu}_k}{| \log \varepsilon_k|} \rightarrow \mu$ in the flat topology.
It remains to show that $\frac{\mu_k - \tilde{\mu}_k}{|\log \varepsilon_k|} \rightarrow 0$ in the flat topology. 
A similar argument can be found in \cite{dLGaPo12}. \\
Let $\varphi \in W^{1, \infty}_0(\Omega;\mathbb{R}^2)$ with $Lip(\varphi) \leq 1$. 
Obviously, $\operatorname{supp } (\mu_k - \tilde{\mu}_k) \subseteq \bigcup_{x \in \operatorname{supp }\mu_k} B_{\varepsilon_k^{\gamma}} (x)$ and by definition $\tilde{\mu}_k = 0$ outside $U_{\varepsilon_k}$. 
Hence,
\begin{equation}\label{eq: flatconvergence}
\int_{\Omega} \varphi \, d(\mu_k - \tilde{\mu}_k) =  \int_{\operatorname{supp } (\mu_k) \setminus U_{\varepsilon_k}} \varphi \,  d\mu_k + \sum_{j \in J_{\varepsilon_k}} \int_{A^j_{\varepsilon_k}} \varphi \,d(\mu_k - \tilde{\mu}_k).
\end{equation}
Next, notice that as the energy $F_{\varepsilon_k}(\mu_k,\beta_k)$ is uniformly bounded and the non-zero elements in $\mathbb{S}$ are bounded away from zero, the number of dislocations is bounded in terms of $|\log \varepsilon|^2$.
Hence, each connected component of $\bigcup_{x \in \operatorname{supp }\mu_k} B_{\varepsilon_k^{\gamma}} (x)$ has a diameter less than $C |\log \varepsilon_k|^2 2 \varepsilon^{\gamma}_k$ where $C$ is a universal constant. \\
We start with the first integral on the right hand side of \eqref{eq: flatconvergence}.
Note that if $x \in \operatorname{supp }\mu_k \setminus U_{\varepsilon_k}$, the corresponding connected component of $\bigcup_{x \in \operatorname{supp }\mu_k} B_{\varepsilon_k^{\gamma}} (x)$ intersects $\partial \Omega$. \\
As $Lip(\varphi) \leq 1$ and $\varphi$ vanishes on $\partial \Omega$, we obtain that $|\varphi| \leq C |\log \varepsilon|^2 2 \varepsilon^{\gamma}_k$ in $B_{\varepsilon_k^{\gamma}} (x)$ and therefore
\begin{equation}
\left| \int_{\operatorname{supp } (\mu_k) \setminus U_{\varepsilon_k}} \varphi \, d \mu_k\right| \leq 2 C |\log \varepsilon_k|^2 \varepsilon^{\gamma}_k \, |\mu_k|(\Omega) \leq \tilde{C}  |\log \varepsilon_k|^4 \varepsilon^{\gamma}_k. \label{eq: flatconv1}
\end{equation}
Next, let us consider $j \in J_{\varepsilon_k}$. 
By \ref{item: lowerbound14} of Proposition \ref{prop: lowerbound1} it holds $\mu_k(A_{\varepsilon_k}^j) = \tilde{\mu}_k(A_{\varepsilon_k}^j)$.
Therefore, 
\begin{equation*}
\int_{A_{\varepsilon_k}^j} \varphi \, d(\mu_k - \tilde{\mu}_k) = \int_{A_{\varepsilon_k}^j} \left(\varphi - <\varphi>_{A_{\varepsilon_k}^j}\right) \, d(\mu_k - \tilde{\mu_k}) ,
\end{equation*}
where $<\varphi>_{A_{\varepsilon_k}^j} = \fint_{A_{\varepsilon_k}^j} \varphi \, dx$.
As $Lip(\varphi)\leq1$, it holds $|\varphi(x) - <\varphi>_{A_{\varepsilon_k}^j}| \leq \operatorname{diam}(A_{\varepsilon_k}^j)$ for all $x \in A_{\varepsilon_k}^j$. 
Thus,
\begin{equation}\label{eq: flatconv2}
\left| \int_{A_{\varepsilon_k}^j} \varphi \, d(\mu_k - \tilde{\mu_k}) \right| \leq C |\log \varepsilon_k|^2 \varepsilon^{\gamma}_k \, (|\mu_k|(A_{\varepsilon_k}^j) + |\tilde{\mu}_k|(A_{\varepsilon_k}^j) ).
\end{equation}
Summing over all $j$ in the estimate \eqref{eq: flatconv2} and combining the resulting estimate with \eqref{eq: flatconv1} yields the existence of a constant $C(\gamma)$ such that
\begin{align*}
\left| \int_{\Omega} \varphi \, d(\mu_k - \tilde{\mu_k}) \right| &= \left| \int_{\operatorname{supp }(\mu_{k}) \setminus U_{\varepsilon_k} } \varphi \, d\mu_k  + \sum_{j \in J_{\varepsilon_k}} \int_{A_{\varepsilon_k}^j} \varphi \, d(\mu_k - \tilde{\mu_k}) \right| \\
&\leq \tilde{C}  |\log \varepsilon_k|^4 \varepsilon^{\gamma}_k + \sum_{j \in J_{\varepsilon_k}} 2 C |\log \varepsilon_k|^2 \varepsilon^{\gamma}_k (|\mu_k|(A_{\varepsilon_k}^j) + |\tilde{\mu}_k|(A_{\varepsilon_k}^j) ) \\
&\leq C(\gamma) |\log \varepsilon_k|^4 \varepsilon^{\gamma}_k.
\end{align*}
Hence, 
\begin{equation*}
\frac1 {|\log \varepsilon_k|} \sup_{\varphi \in W^{1, \infty}_0(\Omega;\mathbb{R}^2): Lip(\varphi) \leq 1} \int_{\Omega} \varphi \, d(\mu_k - \tilde{\mu}_k) \leq C(\gamma) |\log \varepsilon_k|^3 \varepsilon^{\gamma}_k \rightarrow 0 \text{ as } k \to \infty.
\end{equation*}
Thus, we established that $\frac{\mu_k}{|\log \varepsilon_k|} \rightarrow \mu$ in the flat topology which is \eqref{item: comp1}. \\
Note that for all $0 < \gamma < 1$ the corresponding measure $\tilde{\mu}_k$ satisfies $|\tilde{\mu}_k|(\Omega) \leq C(\gamma) |\log \varepsilon_k|$ for some constant $C(\gamma)$ depending on $\gamma$.
A similar argument to the one above shows that the difference of the measures $\tilde{\mu}_k$ for two different values of $\gamma$ converges to zero weakly$*$ in $\mathcal{M}(\Omega;\R^2)$.
Hence, for the chosen subsequence it holds that $\frac{\tilde{\mu}_k}{|\log \varepsilon_k|} \stackrel{*}{\rightharpoonup} \mu$ in $\mathcal{M}(\Omega;\R^2)$ for all $0 < \gamma < 1$ . 
\newline \\
\textbf{Step 2. }\emph{Finding an $L^2_{loc}$-limit for the strains.} \\
Let $0 < \gamma < 1$ and $\alpha = \gamma + \frac{1-\gamma}{2}$.
As $\Omega$ is simply connected, one can find, with the use of Sard's theorem, a monotone sequence of compactly contained subsets $(\Omega_l)_{l \in \mathbb{N}}$ of $\Omega$ such that each $\Omega_l$ is simply connected with Lipschitz boundary and $\Omega_l \nearrow \Omega$. 
In step 1, we applied Proposition \ref{prop: lowerbound1} to $U_{\varepsilon_k}$, $\beta_k$, and $\mu_k$ and obtained a modified strain $\tilde{\beta}_k$ which agrees with $\beta_k$ outside $\bigcup_{x \in \operatorname{supp}\mu_k} B_{\varepsilon_k^{\gamma}}(x)$. 
In particular, $\operatorname{curl } \tilde{\beta}_k = 0$ in $\Omega_{\varepsilon_k^{\gamma}}(\mu_k)$.
Moreover, $|\operatorname{curl } \tilde{\beta}_k|(U_{\varepsilon_k}) \leq C(\gamma) |\log \varepsilon_k|$.
As $\Omega_l \subset \subset \Omega$, it is clear that $\operatorname{dist}(\Omega_l,\partial \Omega) >0$. 
Hence, for $\varepsilon_k$ small enough we find that 
\begin{equation*} %\label{eq: curlboundUl}
\bigcup_{x \in \operatorname{supp }\mu_k } B_{\varepsilon_k^\gamma}(x) \cap \Omega_l \subseteq U_{\varepsilon_k}.
\end{equation*} 
The application of the generalized Korn's inequality (see \cite[Theorem 11]{GaLePo10}) provides a sequence of skew-symmetric matrices $W_k^l$ such that
\begin{equation} \label{eq: kornUl}
\int_{\Omega_l} |\tilde{\beta}_k - W^l_k|^2 \,dx \leq C(\Omega_l) \left(\int_{U_l} |(\tilde{\beta}_k)_{sym}|^2  \, dx + |\operatorname{curl }\tilde{\beta}_k|(\Omega_l)^2 \right).
\end{equation}
From \ref{item: lowerbound1last} of Proposition \ref{prop: lowerbound1} and the bound on $\operatorname{curl} \tilde{\beta}_k$ we derive that $\frac{\tilde{\beta}_k - W^l_k}{|\log \varepsilon_k|}$ is a bounded sequence in the space $L^2(\Omega_l;\mathbb{R}^{2 \times 2})$ where the bound depends on $\gamma$. 
In the following, we use a standard argument to show that the skew-symmetric matrices $W_k^l$ can be chosen independently from $l$. \\
Let us fix $l > 1$.
In addition, let $W_k^l$ and $W_1^l$ be the skew-symmetric matrices from above. 
We may estimate
\begin{equation*}
\mathcal{L}^2(\Omega_1) |W_k^1 - W_k^l|^2 \leq 2 \left( \int_{\Omega_1} |W_k^1 - \tilde{\beta}_k|^2 \, dx + \int_{\Omega_l} |W_k^l - \tilde{\beta}_k|^2 \, dx \right) \leq C(\gamma,l) |\log \varepsilon_k|^2. 
\end{equation*}
Thus, $|W_k^1 - W_k^l| \leq \tilde{C}(\gamma,l) |\log \varepsilon_k|$ which implies that also $\frac{\tilde{\beta}_k - W^1_k}{|\log \varepsilon_k|}$ is a bounded sequence in the space $L^2(\Omega_l;\mathbb{R}^{2 \times 2})$. 
Let us write $W_k = W^1_k$.
As $\beta_k$ agrees with $\tilde{\beta}_k$ on $\Omega_l \cap \Omega_{\varepsilon^{\gamma}_k}(\mu_k)$, we obtain that $\frac{\beta_k - W_k}{|\log \varepsilon_k|} \1_{\Omega_{\varepsilon^{\gamma}_k}(\mu_k)}$ uniformly bounded in $L^2(\Omega_l;\mathbb{R}^{2\times2})$.\\
A similar argument shows that the matrices $W_k$ can also be chosen independently from $\gamma$. \\
Next, let us consider $1 > \gamma_1 > \gamma_2 > 0$. 
Assume that $\frac{\beta_k - W_k}{|\log \varepsilon_k|} \1_{\Omega_{\varepsilon^{\gamma_1}}(\mu_k)} \rightharpoonup \beta$ in $L^2(\Omega_l;\mathbb{R}^{2\times2})$ for some fixed $l$.
Since $\1_{\Omega_{\varepsilon^{\gamma_2}}(\mu_k)} \rightarrow 1$ boundedly in measure and $\1_{\Omega_{\varepsilon^{\gamma_1}}(\mu_k)} \1_{\Omega_{\varepsilon^{\gamma_2}}(\mu_k)} = \1_{\Omega_{\varepsilon^{\gamma_2}}(\mu_k)}$, we deduce that also $\1_{\Omega_{\varepsilon^{\gamma_2}}(\mu_k)} \frac{\beta_k - W_k}{|\log \varepsilon_k|} \rightharpoonup \beta$ in $L^2(\Omega_l;\mathbb{R}^{2\times2})$. \\
On the other hand, it is clear that if for $l_1 > l_2$ we have $\frac{\beta_k - W_k}{|\log \varepsilon_k|} \1_{\Omega_{\varepsilon^{\gamma}_k}(\mu_k)} \rightharpoonup \beta$ in $L^2(\Omega_{l_1};\mathbb{R}^{2\times2})$; the same holds in $L^2(\Omega_{l_2};\mathbb{R}^{2\times2})$. \\
As weak convergence in $L^2$ is metrizable on bounded sets in $L^2$, we can  find by a diagonal argument (in $\gamma$ and $l$) a subsequence and a function $\beta \in L^2_{loc}(\Omega;\mathbb{R}^{2\times2})$ such that $\frac{\beta_k - W_k}{|\log \varepsilon_k|} \1_{\Omega_{\varepsilon^{\gamma}_k}(\mu_k)} \rightharpoonup \beta$ in $L^2(\Omega_l;\mathbb{R}^{2\times2})$ for all $1 > \gamma >0$ and $l \in \mathbb{N}$.
Since $\Omega_l \nearrow \Omega$, this proves the convergence in \eqref{item: comp2}.
In step 4 we show that $\beta \in L^2(\Omega;\mathbb{R}^{2 \times 2})$.
\newline \\
\textbf{Step 3. }\emph{$\operatorname{curl }\beta = \mu$. } \\
Fix some $1 > \gamma > 0$.
In step 1 we saw that $\frac{(\operatorname{curl } \tilde{\beta}_k)_{|U_{\varepsilon_k}}}{|\log \varepsilon_k|} \stackrel{*}{\rightharpoonup} \mu$ in $\mathcal{M}(\Omega;\mathbb{R}^2)$.
A similar argument as the one in \eqref{eq: flatconv1} shows that 
\begin{equation*}
\frac{\operatorname{curl } \tilde{\beta}_k}{|\log \varepsilon_k|} \rightarrow \mu \text{ in the flat topology on }\Omega.
\end{equation*}
This implies convergence in $\mathcal{D}\rq{}(\Omega)$.
On the other hand, we can deduce from step 2 that $\frac{\tilde{\beta}_k - W_k}{|\log \varepsilon_k|} \rightharpoonup \beta$ in $L^2_{loc}(\Omega;\mathbb{R}^{2 \times 2})$ (notice that all arguments were based on considerations for $\tilde{\beta}_k$ and $\tilde{\beta}_k$ equals $\beta_k$ on a set which converges in measure to $\Omega$).
Combining these two facts shows that for $\varphi \in C_c^{\infty}(\Omega;\mathbb{R}^2)$ it holds
\begin{align*}
&< \mu, \varphi >_{\mathcal{D}\rq{},\mathcal{D}} = \lim_k \frac 1{|\log \varepsilon_k|} < \operatorname{curl }(\tilde{\beta}_k - W_k), \varphi >_{\mathcal{D}\rq{},\mathcal{D}} 
\\ = & \lim_k \frac 1{|\log \varepsilon_k|} < \tilde{\beta}_k - W_k, J \nabla \varphi>_{\mathcal{D}\rq{},\mathcal{D}} = < \beta, J \nabla \varphi>_{\mathcal{D}\rq{},\mathcal{D}},
\end{align*}
where $J$ is the clockwise rotation by $\frac{\pi}2$.
Consequently, $\operatorname{curl }\beta= \mu$ which is \eqref{item: comp4}.
\newline \\
\textbf{Step 4. }\emph{The limit $\beta$ is in $L^2(\Omega;\R^{2\times 2})$.} \\
Let $U \subset \subset \Omega$ and $1 > \gamma > 0$ fixed.
From step 1 we know that 
\begin{equation}\label{eq: simplifiedliminf}
\int_U |(\tilde{\beta}_k)_{sym}|^2 \, dx + |(\operatorname{curl }\tilde{\beta}_k)_{|U_{\varepsilon_k}}|(\Omega)^2 \leq C(\gamma) |\log \varepsilon_k|^2.
\end{equation}
Moreover, from step 1 and step 2 we know that 
\begin{equation*}
\frac{(\operatorname{curl }\tilde{\beta}_k)_{|U_{\varepsilon_k}}}{|\log \varepsilon_k|} \stackrel{*}{\rightharpoonup} \mu \text{ in }\mathcal{M}(\Omega;\mathbb{R}^2) \text{ and } \frac{\tilde{\beta}_k - W_k}{|\log \varepsilon_k|} \rightharpoonup \beta \text{ in } L^2(U;\mathbb{R}^{2\times2}).
\end{equation*}
Hence, we derive from \eqref{eq: simplifiedliminf} that
\begin{equation*}
 \int_{U} |\beta_{sym}|^2 \,dx + |\mu|(\Omega)^2 \leq C(\gamma).
\end{equation*}
Taking the supremum over all $U \subset \subset \Omega$ gives 
\begin{equation*}
\int_{\Omega} |\beta_{sym}|^2 \,dx + |\mu|(\Omega)^2 \leq C(\gamma).
\end{equation*}
By the generalized Korn's inequality (see \cite[Theorem 11]{GaLePo10}), there exists a matrix $W \in Skew(2)$ such that $\beta - W \in L^2(\Omega;\mathbb{R}^{2 \times 2})$.
As $\Omega$ has finite measure, this implies that $\beta \in L^2(\Omega;\mathbb{R}^{2 \times 2})$.
\newline \\
\textbf{Step 5. }\emph{Weak convergence of the symmetric part of the strains.} \\
Fix $1 > \gamma > 0$.
As the matrices $W_k$ are skew-symmetric, it is clear that the symmetric part of $(\beta_k - W_k) \1_{\Omega_{\varepsilon_k ^{\gamma}}(\mu_k)}$ is $(\beta_k)_{sym}\1_{\Omega_{\varepsilon_k ^{\gamma}}(\mu_k)}$.
Since taking the symmetric part of a matrix is a linear operation, we may derive from step 2 that 
\begin{equation*}
\frac{(\beta_k)_{sym}}{|\log \varepsilon_k|}\1_{\Omega_{\varepsilon_k ^{\gamma}}(\mu_k)} \rightharpoonup \beta_{sym} \text{ in } L^2(U;\mathbb{R}^{2 \times 2}) \text{ for all } U \subset \subset \Omega.
\end{equation*}
From the bound on the energy we may deduce directly that $\frac{(\beta_k)_{sym}}{|\log \varepsilon_k|}$ is uniformly bounded in $L^2(\Omega;\mathbb{R}^{2\times2})$.
As $(1 - \1_{\Omega_{\varepsilon_k ^{\gamma}}(\mu_k)}) \rightarrow 0$ boundedly in measure, this implies that 
\begin{equation*}
\frac{(\beta_k)_{sym}}{|\log \varepsilon_k|} \rightharpoonup \beta_{sym} \text{ in } L^2(U;\mathbb{R}^{2 \times 2}) \text{ for all } U \subset \subset \Omega. 
\end{equation*}
From the uniform bound of $(\beta_k)_{sym}$ in $L^2(\Omega;\mathbb{R}^{2\times2})$ we conclude \eqref{item: comp3}.
\newline \\
\textbf{Step 6. }\emph{The $\liminf$-inequality.} \newline \\
In step 4, we have already shown a poor man's version of the $\liminf$-inequality. 
For the real $\liminf$-inequality we refer to the proof of the $\liminf$-inequality of the $\Gamma$-convergence result (Proposition \ref{thm: liminf}) in which it can be seen that the convergences established in step 1 and 2 are enough to show that for all $1 > \gamma >0$ it holds that 
\begin{equation*}
\liminf_{k \to \infty} F_{\varepsilon_k}(\mu_k,\beta_k) \geq \int_{\Omega} \mathcal{C} \beta : \beta \,dx + (1 - \gamma) \int_{\Omega} \varphi\left(\frac{d \mu}{d|\mu|} \right) \, d|\mu|.
\end{equation*}
In fact, the estimate for the self-energy works exactly as in the $\liminf$-inequality of the $\Gamma$-convergence result. 
It is computed on the set $U_{\varepsilon_k}$.
For the energy on $\Omega \setminus U_{\varepsilon_k}$ notice that, by step 5, it holds $\1_{U_{\varepsilon_k}} (\beta_k)_{sym} \rightharpoonup \beta_{sym}$ in $L^2(\Omega;\R^{2\times2})$.
Then the estimate follows by classical lower semi-continuity and the fact that $\mathcal{C}$ only acts on the symmetric part of matrices.\\
Sending $\gamma \to 0$ yields the desired $\liminf$-inequality.
\end{proof}
\begin{remark} \label{remark: countercomp}
It can be seen that it is not possible to neglect the reduction to the set $\Omega_{\varepsilon_k^{\gamma}}(\mu_k)$ in the result above.
To construct a counterexample it is enough to arrange dislocations such that $\Omega$ without the corresponding cores of radius $\varepsilon_k$ is disconnected and the total circulation around the disconnected part is $0$.
Then one can define $\beta_k$ to be a constant skew-symmetric matrix $W_k$ on that disconnected piece of the domain.
By choosing $W_k$ large enough, we can guarantee that it does not hold that $\frac{\beta_k - \tilde{W}_k}{|\log \varepsilon_k|}$ is weakly compact in $L^2$ for some sequence $\tilde{W}_k$ of skew-symmetric matrices. 
For an illustration, see figure \ref{figure: remarkcomp}.
\end{remark}

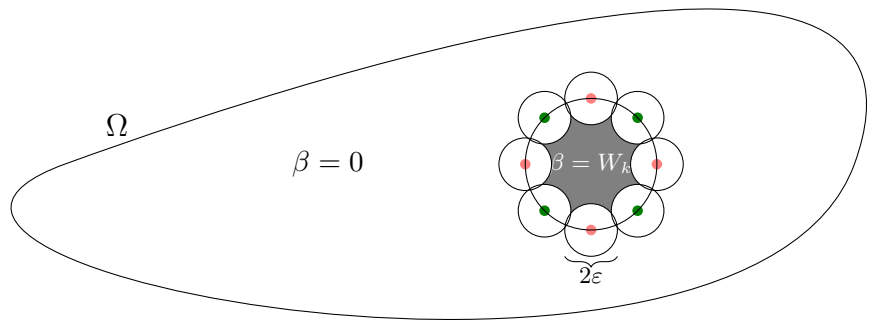
\begin{figure}[h]
\centering
\begin{tikzpicture}[scale= 3.5]
\draw (0.2,0.15) node[font = \large] {$\Omega$};
\draw   (0, 0) to [out=20, in=70]   (3,0)
         to [out=250, in=200] (0,0) ; 
\fill[gray] (2,0) circle(0.25cm);
\fill[white] (2.25,0) circle (0.1cm);
\fill[white] ({2 + 0.25/ sqrt(2)}, {0.25/ sqrt(2)}) circle(0.1cm);
\fill[white] ({2 - 0.25/ sqrt(2)}, {0.25/ sqrt(2)}) circle(0.1cm);
\fill[white] ({2 + 0.25/ sqrt(2)}, {-0.25/ sqrt(2)}) circle(0.1cm);
\fill[white] ({2 - 0.25/ sqrt(2)}, {-0.25/ sqrt(2)}) circle(0.1cm);
\fill[white] (1.75,0) circle (0.1cm);
\fill[white] (2,-0.25) circle (0.1cm);
\fill[white] (2,0.25) circle (0.1cm);
\draw (1,0) node {$\beta = 0$};
\fill[white] (2,0) node[font = \footnotesize] {$\beta =W_k$};

\fill[red!50!white] (2.25,0) circle (0.02cm);
\fill[green!50!black] ({2 + 0.25/ sqrt(2)}, {0.25/ sqrt(2)}) circle(0.02cm);
\fill[green!50!black] ({2 - 0.25/ sqrt(2)}, {0.25/ sqrt(2)}) circle(0.02cm);
\fill[green!50!black] ({2 + 0.25/ sqrt(2)}, {-0.25/ sqrt(2)}) circle(0.02cm);
\fill[green!50!black] ({2 - 0.25/ sqrt(2)}, {-0.25/ sqrt(2)}) circle(0.02cm);
\fill[red!50!white] (1.75,0) circle (0.02cm);
\fill[red!50!white] (2,-0.25) circle (0.02cm);
\fill[red!50!white] (2,0.25) circle (0.02cm);

\draw [decorate,decoration={brace,mirror,amplitude=4pt},xshift=0pt,yshift=-2pt]
(1.9,-0.28) -- (2.1,-0.28)node [black,midway,yshift=-7pt] {\footnotesize
$2\varepsilon$};

\draw (2.25,0) circle (0.1cm);
\draw ({2 + 0.25/ sqrt(2)}, {0.25/ sqrt(2)}) circle(0.1cm);
\draw ({2 - 0.25/ sqrt(2)}, {0.25/ sqrt(2)}) circle(0.1cm);
\draw ({2 + 0.25/ sqrt(2)}, {-0.25/ sqrt(2)}) circle(0.1cm);
\draw ({2 - 0.25/ sqrt(2)}, {-0.25/ sqrt(2)}) circle(0.1cm);
\draw (1.75,0) circle (0.1cm);
\draw (2,-0.25) circle (0.1cm);
\draw (2,0.25) circle (0.1cm);
\draw (2,0) circle(0.25cm);
\end{tikzpicture}
\caption{A possible construction for the configuration discussed in Remark \ref{remark: countercomp}. Here, the green and orange dots represent dislocations with opposite signs.}
\label{figure: remarkcomp}
\end{figure}

\subsection[The $\liminf$-inequality]{The $\liminf$-inequality}\label{section: liminf}

In this section, we prove the $\liminf$-inequality of the $\Gamma$-convergence result in Theorem \ref{theorem: gammanowell}. 
The key ingredients for the lower bound for the part of the energy close to the dislocations will be the Propositions \ref{prop: lowerbound1} and \ref{prop: lowerbound2}.

\begin{proposition}[The \boldmath{$\liminf$}-inequality] \label{thm: liminf}
Let $\Omega \subseteq \R^2$ an open, bounded set with Lipschitz boundary.
Let $\varepsilon_k \to 0$.
Let $(\mu,\beta), (\mu_k,\beta_k) \in \mathcal{M}(\Omega;\mathbb{R}^2) \times L^2(\Omega; \mathbb{R}^{2\times2})$ such that 
\begin{equation*}
\frac{\mu_k}{|\log \varepsilon_k|} \rightarrow \mu \text{ in the flat topology and } \frac{\beta_k}{|\log \varepsilon_k|} \rightharpoonup \beta \text{ in } L^2(\Omega;\R^{2\times 2}).
\end{equation*}
Then it holds
\begin{equation*}
\liminf_{k\to \infty} F_{\varepsilon_k}(\mu_k,\beta_k) \geq F(\mu,\beta).
\end{equation*}
\end{proposition}

\begin{proof}
Clearly, we only have to consider $(\mu_k,\beta_k)$ such that $\liminf_k F_{\varepsilon_k}(\mu_k,\beta_k) < \infty$.
Moreover, up to subsequences we may assume that the $\liminf$ is a $\lim$ and $\sup_k F_{\varepsilon_k}(\mu_k,\beta_k) \leq M < \infty$.
Then the compactness result, Theorem \ref{thm: compactness}, yields that $\operatorname{curl }\beta = \mu$. \\
Next, fix $1 > \alpha > \gamma > 0$. \\
Let us consider the set $U_{\varepsilon_k} = \bigcup_{x \in \operatorname{supp} (\mu_k)} B_{\varepsilon_k^{\gamma}}(x)$.
We split the elastic energy  into a part close to the dislocations and a part far away from the dislocations, precisely
\begin{equation*}
F_{\varepsilon_k}(\mu_k,\beta_k) = \int_{\Omega_{\varepsilon_k^{\gamma}}(\mu_{k})} \frac12 \mathcal{C} \beta : \beta \,dx + F_{\varepsilon_k}(\mu_k,\beta_k,U_{\varepsilon_k}),
\end{equation*}
where we recall that $\Omega_{\varepsilon_k^{\gamma}}(\mu_k) = \Omega \setminus \bigcup_{x \in \operatorname{supp }(\mu_k)} B_{\varepsilon_k^{\gamma}}(x)$.
\newline \\
\textbf{\emph{Lower bound far from the dislocations. }} First, notice that $\1_{\Omega_{\varepsilon_k^{\gamma}}(\mu_{k})} \rightarrow 1$ boundedly in measure. 
As $\frac{\beta_{k}}{|\log \varepsilon_k|} \rightharpoonup \beta$ in $L^2(\Omega;\mathbb{R}^{2\times2})$, this implies $\frac{\beta_{k}}{|\log \varepsilon_k|} \1_{\Omega_{\varepsilon_k^{\gamma}}(\mu_{k})} \rightharpoonup \beta$ in $L^2(\Omega;\mathbb{R}^{2 \times 2})$. 
By the classical lower semi-continuity property of functionals with convex integrands we obtain
\begin{equation*}
\liminf_{k \to \infty} \int_{\Omega_{\varepsilon_k^{\gamma}}(\mu_{k})} \mathcal{C} \beta_k : \beta_k \,dx = \liminf_{k \to \infty} \int_{\Omega} \mathcal{C} \beta_k \1_{\Omega_{\varepsilon_k^{\gamma}}(\mu_{k})} : \beta_k \1_{\Omega_{\varepsilon_k^{\gamma}}(\mu_{k})} \,dx \geq \int_{\Omega} \mathcal{C} \beta : \beta \, dx.
\end{equation*}
This establishes the lower bound of the part of the energy that is not induced by the occurrence of dislocations.
\newline \\
\textbf{\emph{Lower bound close to the dislocations. }}
In this step, we estimate the energy close to the dislocations in terms of the relaxed self-energy density $\varphi$ and the dislocation density $\mu_k$. \\ 
Denote by $(A_{\varepsilon_k}^j)_{j \in J_{\varepsilon_k}}$ the connected components of $U_{\varepsilon_k} = \bigcup_{x \in \operatorname{supp} \mu_{k}} B_{\varepsilon_k^{\gamma}}(x)$ that do not intersect $\partial \Omega$.
We apply Proposition \ref{prop: lowerbound1} to each of  the $A_{\varepsilon_k}^j$, $\beta_k$, $\alpha$, and $\delta = 0$.
For each $j \in J_{\varepsilon_k}$, we obtain a family of balls $(\tilde{B}_i^{j,\varepsilon_k})_{i \in \tilde{I}^j_{\varepsilon_k}}$ and a function $\tilde{\beta}_k: A_{\varepsilon_k}^j \rightarrow \mathbb{R}^{2\times2}$ with the properties \ref{item: lowerbound11} - \ref{item: lowerbound1last} from Proposition \ref{prop: lowerbound1}. 
In particular, the modified strains satisfy 
\begin{equation}
\frac1{|\log \varepsilon_k|^2}\int_{A_{\varepsilon_k}^j} \frac12 \mathcal{C} \tilde{\beta}_k: \tilde{\beta}_k \,dx \leq \left(1 + \frac{C(\alpha,M)}{|\log \varepsilon_k|}\right) F_{\varepsilon_k}(\mu_k,\beta_k,A_{\varepsilon_k}^j). \label{eq: modstrains}
\end{equation}
Fix $\frac15 > \delta > 0$ small enough.
One can check that for $\varepsilon_k > 0$ small enough, for each $j \in J_{\varepsilon_k}$ the sets $A_{\varepsilon_k}^j$, the modified strains $\tilde{\beta}_k$, and the balls $(\tilde{B}_i^{j,\varepsilon_k})_{i \in \tilde{I}^j_{\varepsilon_k}}$ satisfy the assumptions of Proposition \ref{prop: lowerbound2} for $l = \frac12$ and $K=2M$. 
Let us write $\nu_k = \operatorname{curl }\tilde{\beta}_k$ and recall that by \ref{item: lowerbound15} of Proposition \ref{prop: lowerbound1} it holds $\nu_k(A_{\varepsilon_k}^j) = \mu_k (A_{\varepsilon_k}^j)$. \\
The application of Proposition \ref{prop: lowerbound2} yields that for all $\varepsilon_k$ small enough for every $j \in J_{\varepsilon_k}$ we have that at least one of the options in the conclusion of Proposition \ref{prop: lowerbound2} holds.
\newline \\
\underline{\emph{Claim:}} There exists a constant $C(\alpha,M)$ with the following property: For all $\eta>0$ there exists $L \in \mathbb{N}$ such that for all $k\geq L$ it holds for all $j \in J_{\varepsilon_k}$ that 
\begin{equation} \label{eq: claimliminf}
\left(1 + \frac{C(\alpha,M)}{|\log \varepsilon_k|}\right)  F_{\varepsilon_k}(\mu_k,\beta_k,A_{\varepsilon_k}^j) \geq \frac{\alpha - \gamma - \eta - \tilde{\delta}}{|\log \varepsilon_k|} \, \varphi(\mu_k(A^j_{\varepsilon_k})),
\end{equation}
where $\tilde{\delta} = \delta \frac{\max_{\xi \in S^1} \varphi(\xi)}{\min_{\xi \in S^1} \varphi(\xi)}$. 
\newline \\
The strategy to prove this claim will depend on whether the first or second conclusion in Lemma \ref{prop: lowerbound2} holds on $A_{\varepsilon_k}^j$.  \\
Clearly, we may assume that $(\alpha - \gamma - \eta - \tilde{\delta}) > 0$.
\newline \\
\underline{\emph{Case 1: }} Conclusion (ii) of Proposition \ref{prop: lowerbound2} holds for $A_{\varepsilon_k}^j$. \\
Recall (ii) of Proposition \ref{prop: lowerbound2}:
there exists a universal $c>1$ with the following property. 
Consider a ball construction associated to $c$ starting with the balls $(\tilde{B}_i^{j,\varepsilon_k})_{i \in \tilde{I}^j_{\varepsilon_k}}$, which are the output of Proposition \ref{prop: lowerbound1} on $A^j_{\varepsilon_k}$, and the time $s_j^{\varepsilon_k}$ which is defined to be the first time such that a ball in the ball construction intersects $\partial A_{\varepsilon_k}^j$. 
We call its output $(\tilde{I}_{\varepsilon_k}^j(t),(B_i^{j,\varepsilon_k}(t))_{i \in \tilde{I}^j_{\varepsilon_k}(t)},(R_i^{j,\varepsilon_k}(t))_{i \in \tilde{I}^j_{\varepsilon_k}(t)})_t$. 
Then there exists a subset $I_{\varepsilon_k}^j \subseteq \tilde{I}^j_{\varepsilon_k}(s_j^{\varepsilon_k})$ such that for each ball $B_i^{j,\varepsilon_k}(s_j^{\varepsilon_k})$, $i \in I_{\varepsilon_k}^j$, there exist a least $(s_j^{\varepsilon_k} - |\log \varepsilon_k|^{1-\delta} - 1)$ natural numbers $0\leq n_1 < \dots < n_L \leq s_j^{\varepsilon_k}-1$ such that for all $k=1, \dots , L$ no ball in $P^{s_j^{\varepsilon_k}}_i (n_k)$ merges between $n_k$ and $n_k + 1$.
 Moreover, it holds
\begin{equation}
\left| \sum_{i \in \tilde{I}_{\varepsilon_k}^j} \nu_{k}(B_i^{j, \varepsilon_k}) - \nu_k(A_{\varepsilon_k}^j) \right| \leq \delta |\nu_k(A_{\varepsilon_k}^j)|.
\end{equation}
Notice here that $\tilde{\beta}_k$ is $\operatorname{curl }$-free outside the balls $(\tilde{B}_i^{j,\varepsilon_k})_{i \in \tilde{I}^j_{\varepsilon_k}}$. \\
Let $N \in \mathbb{N}$ and define the times $t_l^{\varepsilon_k} = l \frac{s^{\varepsilon_k}_j}{N |\log \varepsilon_k|^{1-\delta}}$ for $l=0, \dots, \lfloor N |\log \varepsilon_k|^{1-\delta} \rfloor$. 
As the starting balls $(\tilde{B}_i^{j,\varepsilon_k})_{i \in \tilde{I}^j_{\varepsilon_k}}$ of the ball construction have radii less than $\varepsilon_k^{\alpha}$ but distance of at least $\frac12 \varepsilon_k^{\gamma}$ to the boundary of $A_{\varepsilon_k}^j$, we can argue as in the proof of Proposition \ref{prop: lowerbound1} to obtain that for $\varepsilon_k$ small enough it holds that $s_j^{\varepsilon_k} \geq (\alpha - \gamma - \frac \eta 4) \frac{|\log \varepsilon_k|}{\log c}$ and consequently
\begin{equation} \label{eq: boundtjeps}
\frac{s_j^{\varepsilon_k}}{N |\log \varepsilon_k|^{1-\delta}} \geq \frac{\alpha - \gamma - \frac{\eta}4}{N} \frac{|\log \varepsilon_k|^{\delta}}{\log c}.
\end{equation}
Next, notice that by Proposition \ref{prop: convpsi} it holds for $\psi_{R,r}(\xi)$ as defined in \eqref{def:psirR} that for all $\xi \in \mathbb{R}^2$ and $R>r>0$ we have that 
\begin{equation*}
\psi_{R,r}(\xi) \geq \log \left(\frac Rr\right) \psi(\xi) - \frac{C |\xi|^2}{\log \left(\frac Rr\right)},
\end{equation*}
where $C>0$ is a universal constant.
Using the $2$-homogeneity and continuity (by convexity) of $\psi$, this implies together with \eqref{eq: boundtjeps} that for $\varepsilon_k$ small enough we may derive for all $\xi \in \mathbb{R}^2$ that 
\begin{equation}\label{eq: estimatepsi}
\psi_{c^{t^{\varepsilon_k}_{l+1}},c^{t^{\varepsilon_k}_l}}(\xi) \geq \left(1-\frac{\eta}4 \right) (t^{\varepsilon_k}_{l+1}-t^{\varepsilon_k}_l) (\log c) \, \psi(\xi).
\end{equation}
Moreover, by the properties of $I_{\varepsilon_k}^j$, it holds true that if $\varepsilon_k$ is small enough, for every $i \in I_{\varepsilon_k}^j$ there exists a subset $J_i^{j,\varepsilon_k} \subseteq \{1, \dots, \lfloor N |\log \varepsilon|^{1 -\delta} \rfloor \}$ such that $\# J_i^{j,\varepsilon_k} \geq (N-2) |\log \varepsilon|^{1-\delta}$ and for each $l \in J_i^{j,\varepsilon_k}$ none of the balls in $P_i^{s_j^{\varepsilon_k}}(t^{\varepsilon_k}_l)$ merges between $t_l^{\varepsilon_k}$ and $t^{\varepsilon_k}_{l+1}$.
For a visualization, see Figure \ref{fig: goodbadmergings}.
Hence, we may estimate for $i \in I_{\varepsilon_k}^j$ (note that the occurring annuli are pairwise disjoint by construction)
\begin{align}
\int_{B_i^{j, \varepsilon_k}(s^{\varepsilon_k}_j)} \frac12 \mathcal{C} \tilde{\beta}_k : \tilde{\beta}_k \,dx & \geq \sum_{l \in J_i^{j,\varepsilon_k}} \sum_{B^{j,\varepsilon_k}_m(t^{\varepsilon_k}_l) \in P_i^{s^{\varepsilon_k}_j}(t^{\varepsilon_k}_l)} \int_{B_m^{j,\varepsilon_k}(t^{\varepsilon_k}_{l+1}) \setminus B^{j,\varepsilon_k}_m(t^{\varepsilon_k}_l)} \frac12 \mathcal{C} \tilde{\beta}_k : \tilde{\beta}_k \, dx \label{eq: deltafirst}\\
&\geq \sum_{l \in J_i^{j,\varepsilon_k}} \sum_{B^{j,\varepsilon_k}_m(t^{\varepsilon_k}_l) \in P_i^{s^{\varepsilon_k}_j}(t^{\varepsilon_k}_l)} \psi_{R^{\varepsilon_k}_m(t^{\varepsilon_k}_{l+1}),R^{\varepsilon_k}_m(t^{\varepsilon_k}_l)} (\nu_k(B^{j,\varepsilon_k}_m(t^{\varepsilon_k}_l))) \nonumber \\
&\geq \sum_{l \in J_i^{j,\varepsilon_k}} \sum_{B^{j,\varepsilon_k}_m(t^{\varepsilon_k}_l) \in P_i^{s^{\varepsilon_k}_j}(t^{\varepsilon_k}_l)} \left(1 - \frac{\eta}4\right) (t^{\varepsilon_k}_{l+1} - t^{\varepsilon_k}_l) (\log c) \, \psi(\nu_k(B^{j,\varepsilon_k}_m(t^{\varepsilon_k}_l))). \nonumber
\intertext{For the last inequality, we used \eqref{eq: estimatepsi} and that by construction it holds $\frac{R^{\varepsilon_k}_m(t^{\varepsilon_k}_{l+1})}{R^{\varepsilon_k}_m(t^{\varepsilon_k}_{l})} = c^{t^{\varepsilon_k}_{l+1} - t^{\varepsilon_k}_{l}}$ . Next, note that $\sum_{B^{j,\varepsilon_k}_m(t^{\varepsilon_k}_l) \in P_i^{s^{\varepsilon_k}_j}(t^{\varepsilon_k}_l)} \nu_k(B^{j,\varepsilon_k}_m(t^{\varepsilon_k}_l)) = \nu_k(B_i^{j,\varepsilon_k}(s_j^{\varepsilon_k}))$. Hence, by the definition of $\varphi$ and the fact that $\nu_k(B^{j,\varepsilon_k}_m(t^{\varepsilon_k}_l)) \in \mathbb{S}$ we can further estimate }
&\geq \left(1 - \frac{\eta}4\right)  \sum_{l \in J_i^{j,\varepsilon_k}} \frac{s^{\varepsilon_k}_j}{N |\log \varepsilon_k|^{1-\delta}} (\log c) \, \varphi(\nu_k(B^{j,\varepsilon_k}_i(s^{\varepsilon_k}_j))) \nonumber \\
&\geq \left(1 - \frac{\eta}4\right)  (N-2) |\log \varepsilon_k|^{1-\delta} \frac{s^{\varepsilon_k}_j}{N |\log \varepsilon_k|^{1-\delta}} \, \varphi(\nu_k(B^{j,\varepsilon_k}_i(s^{\varepsilon_k}_j))) \nonumber \\
&\geq \left(1 - \frac{\eta}4 \right) \frac{N-2}N \left(\alpha - \gamma - \frac \eta 4\right) |\log \varepsilon_k| \, \varphi(\nu_k(B^{j,\varepsilon_k}_i(s^{\varepsilon_k}_j))). \nonumber
\intertext{The simple estimate $\left(1 - \frac{\eta}4\right) x \geq x - \frac{\eta}4$ for $0<x<1$ yields }
&\geq \frac{N-2}N \left(\alpha - \gamma - \frac{\eta}2\right) |\log \varepsilon_k| \, \varphi(\nu_k(B^{j,\varepsilon_k}_i(s^{\varepsilon_k}_j))). \label{eq: deltalast}
\end{align}
\begin{figure}
\centering
\begin{tikzpicture}
\draw (-2.5,0) node[anchor=east] {$B_i^{j,\varepsilon_k}(s_j^{\varepsilon_k})$};
\fill[green!50!gray] (0.5,0) circle (3.0375cm);
\draw (0.5,0) circle (3.0375cm);

\fill[green!50!gray] (0.5,0) circle (2.025cm);
\draw (0.5,0) circle (2.025cm);

\fill[red!40!white] (0.5,0) circle (1.35cm);
\draw (0.5,0) circle (1.35cm);
\draw[dashed] (0.5,0) circle (1cm);

\fill[green!50!gray] (1,0) circle (0.45cm);
\draw (1,0) circle (0.45cm);
\fill[green!50!gray] (0,0) circle (0.45cm);
\draw (0,0) circle (0.45cm);

\fill[gray!70!white] (1,0) circle (0.3cm);
\draw (1,0) circle (0.3cm);
\fill[gray!70!white] (0,0) circle (0.3cm);
\draw (0,0) circle (0.3cm);
\end{tikzpicture}
\caption{Sketch of the situation in case 1 in the proof of Proposition \ref{thm: liminf} for a ball $B_i^{j,\varepsilon_k}(s_j^{\varepsilon_k})$ such that $i \in I^j_{\varepsilon_k}$. The starting balls included in $B_i^{j,\varepsilon_k}(s_j^{\varepsilon_k})$ are drawn in gray. Every dark line marks the state of the ball construction for a certain time $t_l^{\varepsilon_k}$. If there is no merging for a ball included in $B_i^{j,\varepsilon_k}(s_j^{\varepsilon_k})$ between $t_l^{\varepsilon_k}$ and $t_{l+1}^{\varepsilon_k}$ the corresponding annuli are drawn in green, otherwise in red. The dashed line indicates a merging.}
\label{fig: goodbadmergings}
\end{figure}
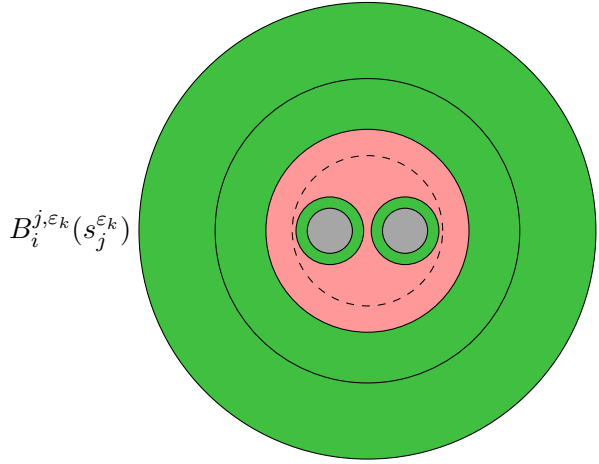
Finally, we choose $N$ so large that $\frac{N-2}N (\alpha - \gamma - \frac{\eta}2) \geq  (\alpha - \gamma - \eta)$. \\
As the family $(B^{j,\varepsilon_k}_i(s^{\varepsilon_k}_j))_{i \in I_{\varepsilon_k}^j}$ consists of pairwise disjoint balls, we can sum over the estimate in \eqref{eq: deltafirst} - \eqref{eq: deltalast} to find that
\begin{align}
&\int_{A_{\varepsilon_k}^j} \frac12 \mathcal{C} \tilde{\beta}_k : \tilde{\beta}_k \, dx \label{eq: estimatephifirst}\\
\geq &(\alpha - \gamma - \eta) |\log \varepsilon_k| \sum_{i \in I_{\varepsilon_k}^j} \varphi(\nu_k(B^{j,\varepsilon_k}_i(s^{\varepsilon_k}_j))) \nonumber  \\
\geq &(\alpha - \gamma - \eta) |\log \varepsilon_k| \, \varphi \left(\sum_{i \in I_{\varepsilon_k}^j} \nu_k(B^{j,\varepsilon_k}_i(s^{\varepsilon_k}_j))\right) \nonumber \\
\geq &(\alpha - \gamma - \eta) |\log \varepsilon_k| \, \left( \varphi(\nu_k(A_{\varepsilon_k}^j)) - \varphi\left(\sum_{i \in I_{\varepsilon_k}^j} \nu_k(B^{j,\varepsilon_k}_i(s_j^{\varepsilon_k})) - \nu_k(A_{\varepsilon_k}^j) \right) \right). \label{eq: estimatephilast}
\end{align}
Here, we used the subadditivity of $\varphi$ for the last and last but second inequality.
Now, note that by the $1$-homogeneity of $\varphi$ and the properties of $I_{\varepsilon_k}^j$ we may derive that 
\begin{align}
\varphi\left(\sum_{i \in I_{\varepsilon_k}^j} \nu_k(B^{j,\varepsilon_k}_i(s^{\varepsilon_k}_j)) - \nu_k(A_{\varepsilon_k}^j) \right) &\leq \left(\max_{\xi \in S^1} \varphi(\xi) \right)  \bigg| \sum_{i \in I_{\varepsilon_k}^j} \nu_k(B^{j,\varepsilon_k}_i(s^{\varepsilon_k}_j)) - \nu_k(A_{\varepsilon_k}^j) \bigg| \label{eq: estimatephi2first} \\
&\leq \left(\max_{\xi \in S^1} \varphi(\xi) \right)  \delta \, |\nu_k(A_{\varepsilon_k}^j)| \nonumber \\
&\leq \delta \, \frac{\max_{\xi \in S^1} \varphi(\xi)}{\min_{\xi \in S^1} \varphi(\xi)} \, \varphi(\nu_k(A_{\varepsilon_k}^j)). \label{eq: estimatephi2last}
\end{align}
Combining \eqref{eq: estimatephifirst} - \eqref{eq: estimatephilast}, \eqref{eq: estimatephi2first} - \eqref{eq: estimatephi2last}, and\eqref{eq: modstrains}, using the inequality $(\alpha-\gamma-\eta)(1-\tilde{\delta}) \geq \alpha-\gamma-\eta-\tilde{\delta}$, and recalling that $\nu_k(A^j_{\varepsilon_k}) = \mu_k(A_{\varepsilon_k}^j)$ proves the claim in the first case. \\ \\
\noindent \underline{\emph{Case 2 :}} $|\mu_k(A_{\varepsilon_k}^j)| = |\nu_k(A_{\varepsilon_k}^j)| \leq |\log \varepsilon|^{1-\delta}$. \\
We only need to consider those $A_{\varepsilon_k}^j$ such that $F_{\varepsilon_k}(\mu_k,\beta_k,A_{\varepsilon_k} j) < \frac{\alpha - \gamma - \eta - \tilde{\delta}}{|\log \varepsilon_k|} \, \varphi(\mu_k(A_{\varepsilon_k}^j))$ (otherwise the desired lower bound is immediate).
The $1$-homogeneity and continuity of $\varphi$ yields that 
\begin{equation*}
F_{\varepsilon_k}(\mu_k,\beta_k,A_{\varepsilon_k}^j) \leq C \, \frac{\alpha - \gamma - \eta - \tilde{\delta}}{|\log \varepsilon_k|} \, |\mu_k(A_{\varepsilon_k}^j)| \leq C (\alpha - \gamma - \eta - \tilde{\delta}) \, |\log \varepsilon_k|^{-\delta} 
\end{equation*}
Hence, we can apply Proposition \ref{prop: lowerbound1} to $\mu_k,\beta_k$, $A_{\varepsilon_k}^j$,  $\alpha,\gamma,\delta$ as fixed before and $K = C$ where $C$ is the universal constant from the estimate above.
Then, we obtain a function $\bar{\beta}_k : A_{\varepsilon_k}^j \rightarrow \mathbb{R}^{2\times2}$ and a family of balls $(D_i^{j,\varepsilon_k})_{i \in I_{\varepsilon_k}^j}$ satisfying the conclusions of Proposition \ref{prop: lowerbound1}.
In particular, the radius of the balls in $(D_i^{j,\varepsilon_k})_{i \in I_{\varepsilon_k}^j}$ is less than $\varepsilon^{\alpha}$,
\begin{equation*}
\# I_{\varepsilon_k}^j \leq K(\alpha) |\log \varepsilon_k|^{1-\delta}, \text{ and } \operatorname{curl } \bar{\beta}_k = 0 \text{ on }A_{\varepsilon_k}^j \setminus \bigcup_{i \in I_{\varepsilon_k}^j} D_i^{j,\varepsilon_k}.
\end{equation*}
Moreover, the strains $\bar{\beta}_k$ satisfy
\begin{equation*}
\frac1{|\log \varepsilon_k|^2} \int_{A_{\varepsilon_k}^j} \frac12 \mathcal{C} \bar{\beta}_k : \bar{\beta}_k \,dx \leq \left(1 + \frac{C(\alpha)}{|\log \varepsilon_k|} \right) F_{\varepsilon_k}(\mu_k,\beta_k,A_{\varepsilon_k}^j).
\end{equation*}
Let us consider a ball construction associated to some $c>1$ not depending on $\varepsilon_k$ or $j$ starting with the balls $(D_i^{j,\varepsilon_k})_{i \in I_{\varepsilon_k}^j}$ as long as for the constructed balls it holds that $B^{j,\varepsilon_k}_i(t) \cap \partial A_{\varepsilon_k}^j = \emptyset$.
As the number of starting balls is bounded by $K(\alpha) |\log \varepsilon_k|^{1-\delta}$, obviously the number of occurring merging times during the ball construction is also bounded by $K(\alpha) |\log \varepsilon_k|^{1-\delta}$.
Hence, we can argue as in case 1 until \eqref{eq: deltafirst} - \eqref{eq: deltalast} to prove the claim also in this case (in this case we do not need the additional $\tilde{\delta}$ on the right hand side of the desired estimate).
\newline  \\
Armed with the statement of  the claim we can now prove the lower bound close to the dislocations. \\
Let $x_{\varepsilon_k}^j \in A_{\varepsilon_k}^j$ and define the measure $\tilde{\mu}_k = \frac{1}{|\log \varepsilon_k|}\sum_{j \in J_{\varepsilon_k}} \mu_k(A_{\varepsilon_k}^j) \delta_{x_{\varepsilon_k}^j}$.
By the statement of the claim and the $1$-homogeneity of $\varphi$, it is clear that $\tilde{\mu}_k$ is a bounded sequence of measures.
Analogously to the proof of the compactness theorem in Section \ref{section: comp} one can show that $\tilde{\mu}_k\stackrel{*}{\rightharpoonup} \mu$ in $\mathcal{M}(\Omega;\mathbb{R}^2)$.
Writing the estimate \eqref{eq: claimliminf} of the claim in terms of $\tilde{\mu}_k$ leads to
\begin{align*}
&\left(1 + \frac{C(\alpha,M)}{|\log \varepsilon_k|}\right) \sum_{j \in J_{\varepsilon_k}} F_{\varepsilon_k}(\mu_k,\beta_k,A_{\varepsilon_k}^j) \geq (\alpha - \gamma - \eta - \tilde{\delta}) \int_{\Omega} \varphi\left( \frac{d \tilde{\mu}_k}{d|\tilde{\mu}_k|} \right) \, d|\tilde{\mu}_k|,
\intertext{which implies that }
&\left(1 + \frac{C(\alpha,M)}{|\log \varepsilon_k|}\right) F_{\varepsilon_k}(\mu_k,\beta_k,U_{\varepsilon_k}) \geq (\alpha - \gamma - \eta - \tilde{\delta}) \int_{\Omega} \varphi\left( \frac{d \tilde{\mu_k}}{d|\tilde{\mu}_k|} \right) \, d|\tilde{\mu}_k|.
\intertext{Then it follows from Reshetnyak's theorem that }
&\liminf_{k \to \infty} F_{\varepsilon_k}(\mu_k,\beta_k,U_{\varepsilon_k}) \geq (\alpha - \gamma - \eta - \tilde{\delta}) \int_{\Omega} \varphi\left( \frac{d \mu}{d|\mu|} \right) \, d|\mu|.
\intertext{Letting $\alpha \to 1, \eta \to 0$, and $\delta \to 0$ yields }
&\liminf_{k \to \infty} F_{\varepsilon_k}(\mu_k,\beta_k,U_{\varepsilon_k}) \geq (1 - \gamma) \int_{\Omega} \varphi\left( \frac{d \mu}{d|\mu|} \right) \, d|\mu|.
\intertext{Combining the bounds far and close to the dislocations we find }
&\liminf_{k \to \infty}  F_{\varepsilon_k}(\mu_k,\beta_k) \geq \int_{\Omega} \frac12 \mathcal{C} \beta: \beta \, dx + (1 - \gamma) \int_{\Omega} \varphi\left( \frac{d \mu}{d|\mu|} \right) \, d|\mu|. 
\end{align*}
Finally, $\gamma \to 0$ finishes the proof of the lower bound.
\end{proof}
%%%%%%%%%%%%%%%%%%%%%%%%%%%%%%%%%%%%%%%
\subsection[The $\limsup$-inequality]{The $\limsup$-inequality}\label{section: limsup}

In this section, we prove the $\limsup$-inequality of the $\Gamma$-convergence result in Theorem \ref{theorem: gammanowell}. 
The proof was mainly worked out in \cite[Theorem 12]{GaLePo10} in the case of well-separated dilocations.
The difference to our setting is that the approximating energies in our case $F_{\varepsilon_k}$ carry the extra term $\frac{|\mu|(\Omega)}{|\log \varepsilon|^2}$.

\begin{proposition}[The \boldmath{$\limsup$}-inequality]
Let $\varepsilon_k \to 0$ and $(\mu,\beta) \in \mathcal{M}(\Omega;\mathbb{R}^2) \times L^2(\Omega;\mathbb{R}^{2 \times 2})$.
There exists $(\mu_k,\beta_k)_k \subseteq \mathcal{M}(\Omega;\mathbb{R}^2) \times L^2(\Omega;\mathbb{R}^{2 \times 2})$ such that 
\begin{enumerate}
\item $\frac{\mu_k}{|\log \varepsilon_k|} \rightarrow \mu \text{ in the flat topology and } \frac{\beta_k}{|\log \varepsilon_k|} \rightharpoonup \beta$ in $L^2(\Omega;\R^{2\times2})$, 
\item $\limsup_{k \to \infty} F_{\varepsilon_k}(\mu_k,\beta_k) \leq F(\mu,\beta) $.
\end{enumerate}
\end{proposition}

\begin{proof}
Let $(\mu_k,\beta_k)_k$ be the recovery sequence from the well-separated case, see \cite[Proof of Theorem 12]{GaLePo10} i.e.,
\begin{enumerate}
\item[(a)] $\frac{\mu_k}{|\log \varepsilon_k|} \stackrel{*}{\rightharpoonup} \mu$ in $\mathcal{M}(\Omega;\R^2)$,
\item[(b)] $\frac{\beta_k}{|\log \varepsilon_k|} \rightharpoonup \beta$ in $L^2(\Omega;\R^{2\times2})$,
\item[(c)] $\limsup_{k \to \infty}\frac1{|\log \varepsilon_k|^2} \int_{\Omega_{\varepsilon_k}(\mu_k)} \frac12 \mathcal{C} \beta_k : \beta_k \, dx \leq F(\mu,\beta).$
\end{enumerate}
It follows directly that $\frac{\mu_k}{|\log \varepsilon_k|} \rightarrow \mu$ in the flat norm. 
Moreover, it follows immediately that
\begin{equation*}
\frac1{|\log \varepsilon_k|^2} |\mu_k|(\Omega) \rightarrow 0.
\end{equation*}
Hence, we find
\begin{equation*}
\limsup_{k \to \infty} F_{\varepsilon_k}(\mu_k,\beta_k) \leq F(\mu,\beta).
\end{equation*}
\end{proof}

\section*{Acknowledgement}

The author is very grateful to Stefan M\"uller and Sergio Conti for bringing the problem to his attention and many fruitful discussions.

\bibliographystyle{abbrv}
\bibliography{no_wellseparateness-arxive}

\begin{thebibliography}{10}
\providecommand{\url}[1]{{#1}}
\providecommand{\urlprefix}{URL }
\expandafter\ifx\csname urlstyle\endcsname\relax
  \providecommand{\doi}[1]{DOI~\discretionary{}{}{}#1}\else
  \providecommand{\doi}{DOI~\discretionary{}{}{}\begingroup
  \urlstyle{rm}\Url}\fi

\bibitem{AldLGaPo14}
Alicandro, R., De~Luca, L., Garroni, A., Ponsiglione, M.: Metastability and
  dynamics of discrete topological singularities in two dimensions: a
  {$\Gamma$}-convergence approach.
\newblock Arch. Ration. Mech. Anal. \textbf{214}(1), 269--330 (2014)

\bibitem{CeLe05}
Cermelli, P., Leoni, G.: Renormalized energy and forces on dislocations.
\newblock SIAM J. Math. Anal. \textbf{37}(4), 1131--1160 (2005)

\bibitem{CoGaMu11}
Conti, S., Garroni, A., M{\"u}ller, S.: Singular kernels, multiscale
  decomposition of microstructure, and dislocation models.
\newblock Arch. Ration. Mech. Anal. \textbf{199}(3), 779--819 (2011)

\bibitem{CoGaOr15}
Conti, S., Garroni, A., Ortiz, M.: The line-tension approximation as the dilute
  limit of linear-elastic dislocations.
\newblock Arch. Ration. Mech. Anal. \textbf{218}(2), 699--755 (2015).
\newblock \doi{10.1007/s00205-015-0869-7}.
\newblock \urlprefix\url{http://dx.doi.org/10.1007/s00205-015-0869-7}

\bibitem{DaLi88}
Dautray, R., Lions, J.L.: Mathematical Analysis and Numerical Methods for
  Science and Technology, vol.~3.
\newblock Springer (1988)

\bibitem{FlHu93}
Fleck, N., Hutchinson, J.: A phenomenological theory for strain gradient
  effects in plasticity.
\newblock J. Mech. Phys. Solids \textbf{41}(12), 1825--1857 (1993)

\bibitem{FrJaMu02}
Friesecke, G., James, R.D., M{\"{u}}ller, S.: A theorem on geometric rigidity
  and the derivation of nonlinear plate theory from three-dimensional
  elasticity.
\newblock Comm. Pure Appl. Math. \textbf{55}(11), 1461--1506 (2002)

\bibitem{GaLePo10}
Garroni, A., Leoni, G., Poniglione, M.: Gradient theory for plasticity via
  homogenization of discrete dislocations.
\newblock J. Eur. Math. Soc. \textbf{12}(5), 1231--1266 (2010)

\bibitem{Gi17}
Ginster, J.: Strain-gradient plasticity as the {$\Gamma$}-limit of nonlinear
  dislocation energy with mixed growth.
\newblock submitted

\bibitem{Gr97}
Groma, I.: Link between the microscopic and mesoscopic length-scale description
  of the collective behavior of dislocations.
\newblock Phys. Rev. B \textbf{56}, 5807--5813 (1997).
\newblock \doi{10.1103/PhysRevB.56.5807}.
\newblock \urlprefix\url{https://link.aps.org/doi/10.1103/PhysRevB.56.5807}

\bibitem{Gu02}
Gurtin, M.E.: A gradient theory of single-crystal viscoplasticity that accounts
  for geometrically necessary dislocations.
\newblock J. Mech. Phys. Solids \textbf{50}(1), 5--32 (2002)

\bibitem{Je99}
Jerrard, R.L.: {Lower bounds for generalized {G}inzburg-{L}andau functionals}.
\newblock J. Math. Anal. \textbf{30}(4), 721--746 (1999)

\bibitem{LaLu16}
Lauteri, G., Luckhaus, S.: An energy estimate for dislocation configurations
  and the emergence of cosserat-type structures in metal plasticity.
\newblock https://arxiv.org/abs/1608.06155.
\newblock \urlprefix\url{https://arxiv.org/abs/1608.06155}

\bibitem{dLGaPo12}
de~Luca, L., Garroni, A., Ponsiglione, M.: {$\Gamma$}-convergence analysis of
  systems of edge dislocations: the self energy regime.
\newblock Arch. Ration. Mech. Anal. \textbf{206}(3), 885--910 (2012)

\bibitem{MuScZe14}
M{\"{u}}ller, S., Scardia, L., Zeppieri, C.I.: Geometric rigidity for
  incompatible fields and an application to strain-gradient plasticity.
\newblock Indiana Univ. Math. J. \textbf{63}(5), 1365--1396 (2014)

\bibitem{Ny53}
Nye, J.: Some geometrical relations in dislocated crystals.
\newblock Acta Metall. \textbf{1}(2), 153--162 (1953)

\bibitem{Or34}
Orowan, E.: Zur {K}ristallplastizit{\"a}t. {III}.
\newblock Zeitschrift f{\"u}r Physik \textbf{89}(9), 634--659 (1934).
\newblock \doi{10.1007/BF01341480}.
\newblock \urlprefix\url{http://dx.doi.org/10.1007/BF01341480}

\bibitem{PaWe61}
Payne, L., Weinberger, H.: On {K}orn's inequality.
\newblock Arch. Ration. Mech. Anal. \textbf{8}(1), 89--98 (1961)

\bibitem{Po34}
Polanyi, M.: {\"U}ber eine {A}rt {G}itterst{\"o}rung, die einen {K}ristall
  plastisch machen k{\"o}nnte.
\newblock Z. Phys. \textbf{89}(9-10), 660--664 (1934)

\bibitem{Po07}
Ponsiglione, M.: Elastic energy stored in a crystal induced by screw
  dislocations: from discrete to continuous.
\newblock SIAM J. Math. Anal. \textbf{39}(2), 449--469 (2007)

\bibitem{Sa98}
Sandier, E.: Lower bounds for the energy of unit vector fields and
  applications.
\newblock J. Funct. Anal. \textbf{152}(2), 379--403 (1998)

\bibitem{ScZe12}
Scardia, L., Zeppieri, C.I.: Line-tension model for plasticity as the
  {$\Gamma$}-limit of a nonlinear dislocation energy.
\newblock SIAM J. Math. Anal. \textbf{44}(2), 2372--2400 (2012)

\bibitem{Ta34}
Taylor, G.I.: The mechanism of plastic deformation of crystals. part {I}.
  theoretical.
\newblock Proc. R. Soc. London, Ser. A \textbf{145}(855), 362--387 (1934)

\bibitem{Vo07}
Volterra, V.: Sur l'{\'e}quilibre des corps {\'e}lastiques multiplement
  connexes.
\newblock Ann. Sci. {\'E}cole Norm. Sup. \textbf{24}(3), 401--517 (1907)

\end{thebibliography}


%\begin{acknowledgements}
%\end{acknowledgements}


\end{document}